\documentclass[reqno]{amsart}
\usepackage{amssymb, amsmath, mathrsfs, amsthm, enumerate, hyperref, mathtools, footmisc, chngcntr}
\usepackage[foot]{amsaddr}
\usepackage[shortlabels]{enumitem}
\setlist[enumerate, 1]{1. }

\numberwithin{equation}{section}

\newcommand{\legendre}[2]{\left(\frac{#1}{#2}\right)}

\theoremstyle{definition}
\newtheorem{define}{Definition}[section]
\newtheorem{thm}[define]{Theorem}
\newtheorem{lem}[define]{Lemma}
\newtheorem{rmk}[define]{Remark}
\newtheorem{alg}[define]{Algorithm}
\newtheorem{prop}[define]{Proposition}
\newtheorem{cor}[define]{Corollary}

\newtheorem{obs}[define]{Observation}

\newtheorem{innercustomgeneric}{\customgenericname}
\providecommand{\customgenericname}{}
\newcommand{\newcustomtheorem}[2]{%
  \newenvironment{#1}[1]
  {%
   \renewcommand\customgenericname{#2}%
   \renewcommand\theinnercustomgeneric{##1}%
   \innercustomgeneric
  }
  {\endinnercustomgeneric}
}

\newcustomtheorem{customthm}{Theorem}
\newcustomtheorem{customlemma}{Lemma}

\begin{document}

\title{Elliptic curves of Fibonacci prime order over $\mathbb{F}_p$}

\author{Rosina Campbell}

\author{Duc Van Huynh$^1$}
\thanks{$ ^1$Corresponding author}

\author{Tyler Melton \and Andrew Percival}

\email{rc2283@stu.armstrong.edu,duc.huynh@armstrong.edu} 
\email{tm3746@stu.armstrong.edu, ap8822@stu.armstrong.edu}

\address{Department of Mathematics, Armstrong State University, Savannah, Ga 31419}

\date{\today}

\begin{abstract}
We will describe an algorithm to construct an elliptic curve $E_{f_q}$ over some prime field $\mathbb{F}_p$ such that such that $|E_{f_q}(\mathbb{F}_p)| = f_q$, where $f_q$ is a probable Fibonacci prime for some prime index $q$. The algorithm is a variant of the efficient CM-construction by Br$\ddot{o}$ker and Stevenhagen, which is well suited for Fibonacci primes due to their arithmetic properties. The time complexity of our algorithm is expected to be lower than $\widetilde{O}(\log^3({f_q}))$. The construction process is a series of algorithms, where each is a test for primality.
\end{abstract}

\maketitle

\section{Introduction}
Let $p > 3$ be a rational prime. Henceforth, for each elliptic curve $E$ over $\mathbb{F}_p$, we say that the order of $E$ is $|E(\mathbb{F}_p)|$, that is, the number of $\mathbb{F}_p$-rational points on $E$. There exists an elliptic curve $E$ of order $N$ for each integer $N$ in the Hasse interval $H_p = [p+1-2\sqrt{p},p+1+2\sqrt{p}]$ \cite[Theorem 14.18]{cox}. Note that $N \in H_p$ if and only if $p \in H_N$, which is a motivation behind the algorithm in \cite{broker2}. Hence, the construction of $E$ is possible exactly when $H_N$ contains a prime $p$. Under General Riemann Hypothesis (GRH), we can safely assume the existence of a prime $p$ in $H_N$ (see \cite{broker}).

This paper is on the study of constructing elliptic curves of Fibonacci prime order over finite fields. A Fibonacci prime is a Fibonacci number that is also prime. It is not known whether there is an infinite number of Fibonacci primes, though heuristics regarding elliptic divisibility sequence (EDS) from \cite{eds} suggests it may be finite. Constructing elliptic curves of Fibonacci order is of interest due to the fact that Fibonacci numbers grow exponentially, and the large width of the Hasse interval $H_{f_q}$ is expected to contain many primes, and Conjecture 4 of \cite{fibform} suggests that the time complexity of the construction may be smaller than it is for other primes. We will also see that the arithmetic properties of Fibonacci primes make some of the computations relatively easier. Furthermore, the construction process allows us to test its primality along the way, with the Elliptic Curve Primality Proving (ECPP) being the main test. \hfill \break

\noindent \textbf{Acknowledgment} This work is a year-long undergraduate research project with the students Rosina Campbell, Tyler Melton, and Andrew Percival. We were partially supported by the Armstrong State University Summer Research Session Grant and by the Armstrong Active Learning Grant. We would like to thank Ayman Bagabas, Chanukya Badri, Donald Hinton, and Keyur Patel for helpful discussion on computing square roots in finite field. Finally, we are grateful of Armstrong's Center for Applied Cyber Education for their continual support.
\section{Main algorithm}

Henceforth, we will use the following notations. We will denote the natural logarithm by $\log$ and for convenience we will write $\log^r(x) = (\log(x))^r$ for real numbers $r > 0 $. The notation $O(x)$ denotes the standard big O notation, and the notation $\widetilde{O}(x)$ means logarithmic terms in $x$ are disregarded.

Given a finite field $\mathbb{F}_p$, we will assume that all computations in $\mathbb{F}_p$ are done using the best known methods. For example, for fast multiplication we will assume the Fast Fourier Transform (FFT) method of \cite{fastmult}, which has time complexity $\widetilde{O}(\log(p))$. The fast multiplication technique is applicable here since the probable Fibonacci primes are at least $2^{50,000}$. For fast  exponentiation, we will assume the method of exponentiation by squaring (see \cite[ch. 1]{cohen1}), which has time complexity $\widetilde{O}(\log^2(p))$ if combined with the fast multiplication method. Similarly, using FFT, the time complexities of multiplication and exponentiation in $\mathbb{F}_p[X]$ are the same as in $\mathbb{F}_p$ for polynomials of small degrees.

Let $f_q$ be a probable Fibonacci prime. Here, we say that $f_q$ is a probable prime if at minimum $q$ is a prime (See Lemma \ref{primeindex}). Our work is based on the list of probable Fibonacci primes given at \cite{online}. We assume that these Fibonacci numbers have been tested under various algorithms, so we do not expect the algorithm to fail before Step \ref{sdisc}. Here in this work we try to construct an elliptic curve of order $f_q$ and test the primality of $f_q$ along the way. The guiding principle behind Algorithm \ref{mainalg} is to interpret all computations as primality tests, though some may be primitive, and we try to reuse all computations whenever possible.

\begin{alg} \label{mainalg} Let $f_q$ be a probable Fibonacci prime. This algorithm attempts to construct an elliptic curve $E/(\mathbb{Z}/p\mathbb{Z})$ of order $f_q$ over some ring $\mathbb{Z}/p\mathbb{Z}$ and performs multiple primality verifications along the way.
\begin{enumerate}[nolistsep]
\item \label{sdensity} Apply the Density Test \ref{densitytest}.
\item From Step \ref{sdensity} we obtain the set $P_q$ of primes $\ell < 2\log(f_q)$ such that $\legendre{f_q}{\ell} = 1$, and we also obtain the first prime $n$ such that $n$ is a quadratic non-residue modulo $f_q$. This $n$ may be greater than $2\log(f_q)$.
\item Use Algorithm \ref{goodD} to obtain a list $S_q$ of good discriminants, and let $N = |S_q|$.
\item \label{spre} Use $n$ to perform square root precomputations using Algorithm \ref{pre}. 
\item Verify that $f_{(q+1)/2}/f_{(q-1)/2} \pmod{f_q}$ is a square root of $-1 \pmod{f_q}$.
\item Apply the Exceptional Cases Test \ref{excases}.
\item Let $k=0$.
\item \label{sdisc} If $k = N$, then go to Step \ref{selkies}, else take $D = S_q[k] \in S_q$ and find a square root $\sqrt{D} \pmod{f_q}$ of $D$ using Algorithm \ref{ts}. Here if $D$ consists of two primes $\ell_1$ and $\ell_2$, then use the previously computed $\sqrt{\ell_1} \pmod{f_q}$ and $\sqrt{\ell_2} \pmod{f_q}$.
\item \label{ssplit} Apply Algorithm \ref{corn} to determine if $f_q$ split completely in $H_K$, the Hilbert class field of $K=\mathbb{Q}(\sqrt{D})$. If $f_q$ does split completely, we obtain $4f_q = x^2 + y^2|D|$ for some positive integers $x,y$. If this step is not successful, increase $k$ by $1$ and return to Step \ref{sdisc}.
\item \label{smodular} Precompute the classical modular polynomials $\Phi_{\ell}(X,Y)$ for primes $\ell < 6\log^2(4\log^2(f_q))$.
\item \label{secpp} Let $p=f_q+1 \pm x$. If it is easy to recognize that $p = k\eta$ for some prime $\eta > (f_q^{1/4}+1)^2$, then construct $E/(\mathbb{Z}/f_q\mathbb{Z})$ of order $f_q+1\pm x$ following Step \ref{spoly} and Step \ref{sconstruct} and apply ECPP (Theorem 13.4) to test the primality of $f_q$, else go to the next step.
\item \label{srm} Apply the Rabin-Miller Primality Test (Algorithm \ref{rm}) to $p$. If $p$ passes the test, then go the next step, else increase $k$ by $1$ and return to Step \ref{sdisc}.
\item \label{spoly} Compute $H_D(X) \pmod{p}$ using Algorithm \ref{crt}.
\item \label{sconstruct} Find a root $r \neq 0,1728$ of $H_D(X) \pmod{p}$. If no root is found, increase $k$ by $1$ and return to Step \ref{sdisc}, otherwise construct the curve
\begin{equation}
E: Y^2 = X^3 +aX-a,
\end{equation}
where 
\begin{equation}
a = \frac{27r}{4(1728-r)} \pmod{p}.
\end{equation}
\item Let $\mathcal{E}_q$ be an empty list.
\item \label{scheckorder} Append one of $(E,p,D)$ or $(E_{\text{twist}},p,D)$ to the list $\mathcal{E}_q$ for which $f_q \cdot (1,1) = 0$ is satisfied. If neither of the twists satisfy $f_q \cdot (1,1) = 0$, then increase $k$ by $1$ and return to Step \ref{sdisc}, else go to the next step.
\item If $f_q$ is a confirmed prime, then $p$ is prime as well by ECPP. Output $(E,p,D)$ and stop the algorithm. If the primality of $f_q$ is not confirmed, then increase $k$ by $1$ and return to Step \ref{sdisc}.
\item \label{selkies} Apply Algorithm \ref{elkies} (Elkies Primes Verification) to the list $\mathcal{E}_q$. If $\mathcal{E}_q$ is empty, then $f_q$ is likely composite, else go to the next step.
\item \label{seigenvalue} Apply Algorithm \ref{eigenvalue} (Eigenvalue Verification) to the list $\mathcal{E}_q$. If $\mathcal{E}_q$ is empty, then $f_q$ is likely composite, else go to the next step.
\item Output a random element $(E,p,D)$ from $\mathcal{E}_q$.
\end{enumerate}
\end{alg}

A small issue with Algorithm \ref{mainalg} is the primality of $p$. By ECPP (Theorem \ref{ecpp}), the verification $f_q \cdot (1,1) = 0$ in Step \ref{sconstruct} does confirm that $p = f_q+1\pm x$ is a prime if $f_q$ is known to be prime. However, $f_q$ is a probable Fibonacci prime that we wish to test the primality of.

If $p$ is confirmed to be prime, then $f_q$ is automatically prime by ECPP in Step \ref{secpp} and we have an elliptic curve $E/\mathbb{F}_p$ of order $f_q$. Even when $p$ fails to be prime, we can still apply ECPP to determine the primality of $f_q$. Explicitly, if $p = k\eta$ for some recognizable prime $(f_q^{1/4}+1)^2 < \eta < p$ , then we have that $f_q$ is a prime assuming that $k\cdot (1,1)$ is defined and not equal to $0$. If it is not easy to confirm the primality of $\eta$, then we move on to another discriminant. This is a common technique in applying ECPP (see \cite[pp. 597]{handbook}).

 The philosophy behind Step \ref{selkies} and Step \ref{seigenvalue} is to prolong computations in $\mathbb{Z}/p\mathbb{Z}$ to detect the compositeness of $p$, and we want to verify the order of each curve $E/(\mathbb{Z}/p\mathbb{Z})$ as well. We will use Schoof's algorithm to verify $t \pmod{\ell}$ for some small Elkies primes $\ell$ relative to $f_q$. In fact, the largest probable Fibonacci prime from the list \cite{online} has index $2904353$, so we need to compute $\Phi_{\ell}(X,Y)$ for Elkies primes $\ell < 5287$. Of course if storage capacity is not limited, then it is practical to precompute the modular polynomials $\Phi_{\ell}(X,Y)$ for large Elkies primes $\ell$ if multiple probable Fibonacci primes are to be tested. Note here that we can not use modular polynomials for Weber's function since $D \equiv 5 \pmod{8}$. A combination of isogeny volcanoes with other class invariant such as Ramanujan's class invariant(see \cite{konto}) can allow one to work with larger Elkies primes, but such implementation is beyond our capacity. See \cite{as5} for computation of the the class modular polynomials $\Phi_{\ell}(X,Y)$ for large primes $\ell$ via isogeny volcanoes, which suggests it is best to compute $\Phi_{\ell}(X,Y) \pmod{p}$ as needed.

\begin{thm} \label{complexity} Assuming GRH, the time complexity of the Algorithm \ref{mainalg} is $\widetilde{O}(\log^3(f_q))$. Furthermore, the space required is of size $\widetilde{O}(\log^2(f_q))$.
\end{thm}
\begin{proof}
Algorithm \ref{mainalg} is similar to the algorithm from \cite{broker2}, which has time complexity $\widetilde{O}(\log^3(f_q))$. The only difference here is that our algorithm is a bit more convoluted and we have the extra final verifications using Schoof's algorithm in Step \ref{selkies} and Step \ref{seigenvalue}, which have time complexity $\widetilde{O}(\log^2(f_q))$ and $\widetilde{O}(\log^3(f_q))$, respectively.

Each step has time complexity of at most $\widetilde{O}(\log^2(f_q))$. Even though our algorithm has $O(\log^2(f_q))$ loops, we will show that the steps that have time complexity $\widetilde{O}(\log^2(f_q))$ get called only  $O(\log(f_q))$ times.

The square root algorithm (Algorithm \ref{ts}) get called only in the case $D = -\ell$, which happens $O(\log(f_q))$ times. Since each square root computation takes time $\widetilde{O}(\log^2(f_q))$, the total time on computing square roots is $\widetilde{O}(\log^3(f_q))$.

By the Chebotarev Density Theorem, it is expected to find $O(\log(f_q))$ discriminants $D$ for which $4f_q = x^2 + y^2|D|$ for some positive integer $x,y$; we will see this fact in Section \ref{efficient}. Since the time complexity of each step from Step \ref{ssplit} to Step \ref{scheckorder} is at most $\widetilde{O}(\log^2(f_q))$, we have a total time complexity of $\widetilde{O}(\log^3(f_q))$ if we loop through all such discriminants. Therefore, the entire algorithm has time complexity $\widetilde{O}(\log^3(f_q))$.

The main step in Algorithm \ref{mainalg} requiring the most storage is in computing $H_D(X) \pmod{p}$. By \cite{as2}, the storage needed to computed $H_D(X) \pmod{p}$ is $O(|D|^{1/2 + \epsilon} \log(p))$. As $D = O(\log(f_q)^2)$ and $p = O(f_q)$, it follows that the storage required for Algorithm \ref{mainalg} is $\widetilde{O}(\log^2(f_q))$.

\end{proof}

It is of small concern (or of great fortune) if Conjecture 4 of \cite{fibform} is true. Let $d$ be an integer, and let $\mathcal{N}_d$ be the set of positive integers defined by
\begin{equation}
\mathcal{N}_d = \{n > 0: f_n = |x^2+dy^2| \text{ for some integers $x$ and $y$} \}.
\end{equation}
The lower asymptotic density $\underline{\delta}(\mathcal{N}_d)$ of $\mathcal{N}_d$ is defined by
\begin{equation}
\underline{\delta}(\mathcal{N}_d) = \liminf_{n \rightarrow \infty} \frac{1}{n} \sum_{k=1}^n \chi_{\mathcal{N}_d}(k) = \liminf_{n \rightarrow \infty} \frac{\mathcal{N}_d \cap [1,n]}{n},
\end{equation}
where $\chi_{\mathcal{N}_d}$ is the characteristic function of the set $\mathcal{N}_d$. The conjecture states that the lower asymptotic density $\underline{\delta}_{\mathcal{N}_d}$ is $0$ for all but finitely many integers $d$ not a square or negative of a square.

For a Fibonacci prime $f_q$ and a positive square-free integer $d$, we have $f_q = x^2 + dy^2$ for some integers $x,y$ exactly when $f$ splits completely in the ring class field of the imaginary quadratic field $K=\mathbb{Q}(\sqrt{-d})$. Here the ring class field is the extension $R_K/K$ corresponding to the ideal class group $C(\mathbb{Z}[\sqrt{-d}])$ given by Class Field Theory. The extension $R_K/H_K$ is of degree $2$ exactly when $-d \equiv 1 \pmod{4}$ and degree $1$ when $-d \equiv 3 \pmod{4}$, where $H_K$ is the Hilbert class field corresponding to the maximal order $\mathcal{O}_K$.

Let $h_d$ be the class number of $\mathcal{O}_K$. By the Chebotarev Density Theorem, the density of the rational primes that split completely in $H_K$ is $1/2h_d$, and half of those split completely in $R_K$, that is, $1$ out $4h_d$ rational primes split completely in $R_K$. The conjecture implies that if $d$ is large enough, only a few Fibonacci primes split completely in $R_K$, at best it is an infinite set of density $0$, which implies that the Fibonacci primes splitting completely in $H_K$ has lower asymptotic density $0$ as well. This may pose difficulty to Algorithm \ref{mainalg} due to its reliance on finding a Hilbert class field $H_K$ for which $f_q$ splits completely. On the other hand, the conjecture also suggests that the Fibonacci primes only split completely in Hilbert class fields (induced by \textit{small} discriminant). This is fortunate as we wish to find a small fundamental discriminant $D$ for which $4f_q = x^2 + y^2|D|$ for some integers $x,y$ such that $f_q + 1 \pm x$ is a prime, so Step \ref{ssplit} has a higher chance of success. Hence, the time complexity may be lower than $\widetilde{O}(\log^3(f_q))$. 

\section{Overview} 
In Section \ref{cm}, we will see that one could in theory construct an elliptic curve of order $f_q$ by picking a prime $p$ in the Hasse interval $H_{f_q}$ and finding a root of $H_D(X) \pmod{p}$, where $D = (p+1-f_q)^2 - 4p$. Of course, this is highly impractical as the discriminant $D$ may be too large, and there does not exists an efficient method to find the fundamental discriminant induced by $D$. In \cite{broker2}, it is observed that 
\begin{equation}
D = (p+1-f_q)^2 - 4p = (f_q+1-p)^2-4f_q.
\end{equation}
This observation allows us to construct a suitable fundamental discriminant from a basis of primes (Section \ref{efficient}). This is a small discriminant that induces a class field in which both of the primes $p$ and $f_q$ split completely.

We will be mainly implementing the complex multiplication method for the construction, and we will use the algorithm from \cite{broker2} to find a small discriminant. One major hurdle of \cite{broker2} in constructing an elliptic curve of prime order $N$ is to compute a square root of $(-1)^{\frac{\ell-1}{2}}\ell$ modulo $N$ for various primes $\ell$,  but we will see that for each probable Fibonacci prime $f_q$ it is a straightforward application of the Tonelli-Shanks algorithm because the 2-Sylow subgroup of $\mathbb{Z}/(f_q-1)\mathbb{Z}$ is small.

We will provide in Section \ref{cm} a brief overview of complex multiplication and its application toward constructing elliptic curves of prescribed torsion. In Section \ref{efficient} we will described the algorithm from \cite{broker2} to efficiently find small discriminant.  We will provide a quick overview of the CRT method in Sections \ref{volcano} to  \ref{classgroup} following \cite{as2}. We will go over computing square roots in finite field and Cornacchia's algorithm in Section \ref{roots}.  In Section \ref{schoof} we will discuss Schoof's algorithm and provide a way to verify that we have the correct curve. We will go over some elementary properties of Fibonacci numbers in Section \ref{fibonacci}. We will look into equations of the form $4f_q = x^2+dy^2$ for special $d$'s in Sections \ref{exceptional}. Finally, in Section \ref{bettertests} we will discuss the primality tests that are used in our algorithm.

\section{Complex multiplication and application} \label{cm}
In this section, we will provide an overview of complex multiplication and its application in constructing elliptic curves of prescribed order. For further discussion on complex multiplication see \cite{schoof}, \cite{am}, \cite[pp. 190--196]{cox}, \cite[pp.95-100]{silverman2}, \cite[pp. 455-460]{handbook}, and \cite{chenal}

Henceforth, let $D$ be a negative discriminant. The polynomial $F(x,y) = aX^2+bXY+cY^2$ is called a binary quadratic form, where $a,b,c \in \mathbb{Z}$. We say that the form $F$ is reduced if $\gcd(a,b,c)=1$, and $a,b$, and $c$ satisfy the condition
\begin{equation}
|b| \leq |a| \leq c \text{ and } b \geq 0 \text{ whenever } |b| = a \text{ or } a = c.
\end{equation}
The discriminant of $F$ is defined by $D = b^2 - 4ac$.

To each form $F$, we associate the matrix
\begin{equation}
M_F = \begin{bmatrix}
       a & \frac{b}{2} \\
       \frac{b}{2} & c 
     \end{bmatrix}.
\end{equation}
Two forms $F_1,F_2$ are said to be equivalent if there exists a matrix $N \in \text{SL}_2(\mathbb{Z})$ such that

\begin{equation}
M_{F_2} = N^{-1} M_{F_1} N.
\end{equation}

The relation defines an equivalence relation on quadratic forms. The set of equivalence classes $C(D)$ of quadratic forms with discriminant $D$ forms an abelian group, and each class contains exactly one reduced form by Theorem 2.8 of \cite{cox}. We will call $C(D)$ the class group induced by  discriminant $D$. 

Let $K = \mathbb{Q}(\sqrt{d})$ for some rational integer $d < 0$, and let $O_K$ be its ring of integers. Let $\mathcal{O} \subset O_K$ be an order of index $f$, which is called the conductor of $\mathcal{O}$. The discriminant of $\mathcal{O}$ is $D = \text{Disc}(\mathcal{O}) = f^2d_K$, where $D_K$ is the field discriminant (or fundamental discriminant) of $K$. From \cite[Theorem 5.30]{cox}, we have 
\begin{equation}
C(D) \cong C(\mathcal{O}),
\end{equation}
where $C(\mathcal{O})$ is the ideal class group of $\mathcal{O}$. Explicitly, the isomorphism $C(D) \xrightarrow{\sim} C(\mathcal{O})$ above is given by 
\begin{equation} \label{cgcorr}
aX^2+bXY + cY^2 \mapsto [a,(-b+\sqrt{D})/2],
\end{equation}
where $aX^2+bXY + cY^2$ is a reduced form of discriminant $D$.
This provides an easy way to study the group $C(\mathcal{O})$, and in particular, to compute the class order.

As $C(\mathcal{O})$ is a quotient of the ray class group of conductor $\mathfrak{f} = f\mathcal{O}_K$ of $K$, Class Field Theory (see \cite[Theorem 8.6]{cox}) tells us that there exists a unique abelian extension $L/K$ such that
\begin{equation}
C(\mathcal{O}) \cong \text{Gal}(L/K),
\end{equation}
where the isomorphism is given by the Artin map.

Let $\mathbb{H}$ be the upper half of the complex plane, and let $h$ be the order of $C(D)$. Then the minimal polynomial $H_D(x)$ of $L/K$ is given by
\begin{equation} \label{hilbertpoly}
H_D(x) = \prod_{i=1}^h (x-j(\tau_i)),
\end{equation}
where $\tau_i \in \mathbb{H}$ is a root of $F(x,1)$, the reduced form representing the class $[F(x,y)]$ in $C(D)$, and $j$ is the well-known $j$-invariant function ($\mathbb{C}$-isomorphism)
\begin{equation}
j: X(1) \rightarrow \mathbb{P}^1(\mathbb{C}),
\end{equation}
where $X(1)$ is the modular curve 
\begin{equation}
X(1) = \frac{\mathbb{H} \bigcup \mathbb{P}^1(\mathbb{Q})}{\text{SL}_2(\mathbb{Z})}
\end{equation}
and $\mathbb{P}^1(\mathbb{C})$ is the Riemann sphere.
Moreover, the Fourier series expansions of $j$ begins with
\begin{equation} \label{fourier}
j(\tau) = q^{-1}+744+196884q+2149376q^2 + \dots,
\end{equation}
where $q = e^{2\pi i \tau}$. One of the miracles in Explicit Class Field Theory is that $L = K(j(\tau))$, where $j(\tau)$ is any root of $H_D(X)$ - fulfilling Kronecker's jugendtraum.

The polynomial $H_D(X)$ in Equation \ref{hilbertpoly} is called the Hilbert class polynomial associated with the order $\mathcal{O}$. It seems that it is standard to call $H_D(X)$ a Hilbert class polynomial regardless of whether $\mathcal{O}$ is maximal. We will continue such naming standard.

Computing the Hilbert class polynomial can be quite difficult. Besides the already daunting time complexity in computing $H_D(X)$, the storage required to store its coefficients may be beyond practical purpose. For example, it requires $47.2$ petabytes of storage in constructing $H_D(X)$ for $D = -(10^{16}+135)$ (see \cite{as3}). The complex-analytic method of computing $H_D(X)$ is to approximate each root using the expansion \ref{fourier}, and to verify for accuracy, we use the fact that $H_D(x)$ has integer coefficients, and $\sqrt[3]{H_D(0)} \in \mathbb{Z}$, a consequence of the work of Gross and Zagier \cite[Proposition 7.1]{am}. There are two other known methods to compute $H_D(X)$: the $p$-adic lifting method (see \cite{broker3}), and application of isogeny volcanoes and Chinese Remainder Theorem (see \cite{as2} and its accelerated version \cite{as3}). We will see a quick overview of isogeny volcanoes in Section \ref{volcano}. It is interesting to note that the complex-analytic method has to deal with rounding errors, while the p-adic lifting method circumvent that by working in a non-archimedean setting.

There is a correspondence between representatives of $C(\mathcal{O})$ and $\mathbb{C}$-isomorphism class of elliptic curves with endomorphism ring isomorphic to $\mathcal{O}$ (see \cite[Corollary 10.20]{cox}). Viewing each ideal $\mathfrak{a}$ of $\mathcal{O}$ as a lattice of $\mathbb{C}$, the correspondence is given by
\begin{equation}
\mathfrak{a} \mapsto \mathbb{C}/\mathfrak{a},
\end{equation} 
and in view of the correspondence from \ref{cgcorr}, we have
\begin{equation}
j((a+b\sqrt{D})/2) = j(\mathbb{C}/\mathfrak{a}).
\end{equation}
It is now straightforward to obtain an algebraic model for each curve $\mathbb{C}/\mathfrak{a}$. For each root $r \neq 0,1728$ of $H_D(X)$, let $E_{r}/L$ be the elliptic curve given 
\begin{equation}
Y^2 = X^3 + aX-a,
\end{equation}
where $a = \frac{27r}{4(1728-r)}$. While if $r = 0$, let $E_r/L$ be given by
\begin{equation}
Y^2 = X^3+1,
\end{equation}
and if $r=1728$, let $E_r/L$ be given by
\begin{equation}
Y^2 = X^3+X.
\end{equation}
The elliptic curve  $E_r/L$ has coefficients in  $L$, and its $j$-invariant is $j=r$. Furthermore, the endomorphism ring End$(E_r)$ is isomorphic to $\mathcal{O}$. Henceforth, we will identify End$(E_r)$ with $\mathcal{O}$.

Let $p$ be a rational prime that splits in $K$, and let $\mathfrak{p}$ be a prime ideal of $\mathcal{O}_{L}$ that divides the ideal $(p)$. Assume that $\mathfrak{p} \nmid \Delta(E_r)$, where $\Delta(E_r) = -16(4a^3+27a^2)$ is the discriminant of $E_r$, then $E$ has good reduction at $\mathfrak{p}$. The reduction $\overline{E}$ of $E$ mod $\mathfrak{p}$ has coefficients in some finite extension $\mathbb{F}_q$ of $\mathbb{F}_p$, so in the case that $p$ splits completely in $L$ or the class group $C(D)$ has order $1$, the reduction $\overline{E}$ has coefficients in $\mathbb{F}_p$. The endomorphism ring of $\overline{E}$ is isomorphic to $\mathcal{O}$ and its j-invariant is a root of $H_D(X) \pmod{q}$. Deuring's Reduction Theorem tells us that every elliptic curve over $\mathbb{F}_q$ with endomorphism ring isomorphic to $\mathcal{O}$ arises this way (see \cite[Theorem 14.16]{cox}). Moreover, we have 
\begin{equation}
|\overline{E}(\mathbb{F}_q)|=q+1-t,
\end{equation} 
where $t = \pi + \overline{\pi}$ for some $\pi \in \mathcal{O}$ such that $q = \pi \overline{\pi}$. If we have an element $\beta \in \mathcal{O}$ such $q = \beta \overline{\beta}$, then $\beta/\pi \in \mathcal{O}^{\times}$. As we will be working with $\mathcal{O}$ with discriminant $D < -4$, the group of units $\mathcal{O}^{\times} = \{\pm 1\}$, so $\beta + \overline{\beta}$ may differ from $t$ by a negative sign.

We observe that the roots of $H_D(X)$ are the j-invariants of all elliptic curves $E/L$ with endomorphism ring $\mathcal{O}$. Let Ell$_{\mathcal{O}}(L)$ be the set of all roots of $H_D(X)$. The ideal class group $C(\mathcal{O})$ provides a free transitive group action on the set Ell$_{\mathcal{O}}(L)$. To see the group action, let $\legendre{\boldsymbol{\cdot}}{L/K}: C(\mathcal{O}) \rightarrow \text{Gal}(L/K)$ be the Artin map. For an invertible ideal $\mathfrak{a} \in \mathcal{O}$, we have
\begin{equation} \label{shimura}
\legendre{\mathfrak{a}}{L/K}(j(E)) = j(E/E[\mathfrak{a}]),
\end{equation}
where $E[\mathfrak{a}]$ is the $\mathfrak{a}$-torsion of $E$ (see \cite[ch. 11]{cox}). The fact that this action is free and transitive follows from the fact that $\mathbb{C}/\mathfrak{a}$ determines an isomorphism class of elliptic curves over $L$ with endomorphism ring $\mathcal{O}$ (see \cite{as2}, \cite[Corollary 10.20]{cox} and \cite{broker3} for further details).

As $C(\mathcal{O}) \cong C(D)$ and there is a bijection between Ell$_{\mathcal{O}}(\mathbb{F}_p)$ and Ell$_{\mathcal{O}}(L)$ by Deuring lifting theorem, there is a free transitive group action of $C(D)$ on Ell$_{\mathcal{O}}(\mathbb{F}_p)$. Hence, if we have one root $j_0$ of $H_D(X) \pmod{p}$, we can obtain the rest by computing the orbit of the group action of $C(D)$ on $j_0$. This fact is used in \cite{as2}, which we will see an overview of in Section \ref{volcano}.

As mentioned earlier, the reduction $\overline{E}$ at $\mathfrak{p}$ has coefficients in $\mathbb{F}_p$ exactly when $\mathfrak{p}$ splits completely in $L$ and $p$ splits in $K$, which happens exactly when $H_D(X) \pmod{p}$ splits completely over $\mathbb{F}_p$ (see \cite[Theorem 5.1, Theorem 9.2]{cox}). Since $L/K$ is Galois, $H_D(X) \pmod{p}$ splits completely exactly when $H_D(X) \pmod{p}$ has a root in $\mathbb{F}_p$ for $p \nmid D$. On the other hand, from Class Field Theory the prime ideal $\mathfrak{p}$ splits completely in $L$ exactly when $\mathfrak{p}$ is principal, which happens exactly when the rational prime $p$ is a norm in $\mathcal{O}$, that is, there exists integers $x,y$ such that
\begin{equation}
4p = x^2+y^2|D|.
\end{equation}
Hence, $H_D(X) \pmod{p}$ has a root if and only if the $4p = x^2+y^2|D|$ for some positive integers $x,y$.

The previous paragraphs provide a method to construct elliptic curves over a prime field of prescribed order. Indeed, let $N$ be a positive integer and let $t = p+1-N$, where $p$ is a prime so that $|t| \leq 2\sqrt{p}$, that is, $p$ is in the Hasse's interval $H_N$. Let $D = (p+1-N)^2-4p$, and compute $H_D(X) \pmod{p}$. We find a root $r \neq 0,1728$ in $\mathbb{F}_p$ of $H_D(X) \pmod{p}$, which exists because the equation $4p = X^2+Y^2|D|$ has the solution $(t,1)$. Compute $a = \frac{27r}{4(1728-r)} \pmod{p}$, and consider the elliptic curve $E/\mathbb{F}_p$ defined by $E: Y^2 = X^3+aX-a$. The order of $E$ is $p+1\pm t$, so we may have to compute its quadratic twist (see Proposition 5.4 of \cite{silverman1})
\begin{equation}
E_{\text{twist}}: Y^2 = X^3 + g^2aX-g^3a
\end{equation}
if necessary to find the one with order $N = p+1-t$, where $g$ is any quadratic non-residue modulo $p$. The fact that the point $(1,1)$ lies on $E$ allows us to quickly determine which curve has the correct order. For the cases $r=1728$ or $r=0$, the set of twists of the curves $Y^2=X^3+X$ and $Y^2 = X^3+1$ correspond to $\mathbb{F}_p^*/(\mathbb{F}_p^*)^4$ and $\mathbb{F}_p^*/(\mathbb{F}_p^*)^6$, respectively (see Proposition 5.4 of \cite{silverman1}).

\begin{alg} \label{complex} Complex Multiplication algorithm to construct an elliptic curve of order $N$ over some prime field.
\begin{enumerate}[nolistsep]
\item Find a prime $p \in H_N$.
\item Compute $D = (p+1-N)^2 - 4p$.
\item Compute the Hilbert class polynomial $H_D(X)$.
\item Find a root $r$ of $H_D(X) \pmod{p}$.
\item If $r \neq 0,1728$, construct the curve $E: Y^2 = X^3+aX-a$, where $a = (27r)/(4(1728-r)) \pmod{p}$.
\item If $r = 0$, take $E: Y^2 = X^3+1$, and if $r = 1728$, take $E: Y^2 = X^3+X$.
\item Test the point $(1,1)$ of $E$, that, is verify
\begin{equation}
(p+1)(1,1) = t(1,1).
\end{equation}
\item Compute a twist of $E$ if necessary.
\end{enumerate}
\end{alg}

A careful analysis shows that instead of using $D = (p+1-N)^2 - 4p$, we could use the fundamental discriminant $D_K = D/k^2$, for some integer $k$. Here, fundamental discriminant is equivalent to the field discriminant of $K = \mathbb{Q}(\sqrt{D})$. As the fundamental discriminant of $D$ is essentially its square-free part, computing the fundamental discriminant of $D$ is not practical as computing the square-free part of an integer is exceedingly difficult.

Constructing elliptic curves of prescribed torsion has its applications in cryptography (ECC), mainly in creating keys for encryption systems such as AES. In the coming age (or current age) of quantum computers, a system such as ECC will be (is) breakable. There has been active research into post-quantum cryptography in the last decade to find a system that could withstand a quantum computer. For example, the Supersingular Isogeny Diffie–-Hellman Key Exchange has been shown to be a great candidate (see \cite{super}).

\begin{rmk}
Instead of using the Hilbert class polynomial, one could also use \textit{smaller} class polynomials such as Weber class polynomials (see \cite{classpoly1}) and Ramanujan's class polynomial (see \cite{classpoly2}). Among the well-known class polynomials, the data from \cite{classpoly2} shows that the Ramanujan's class polynomials are best for generating elliptic curves of prime order. One could also in theory find other class invariants with smaller class polynomials using a variant of Shimura Reciprocity (see \cite{gee} and \cite{konto}). Here we say that $\mathfrak{f}(\tau))$ is a class invariant of a Hilbert class field $H_K$ if $H_K = K(\mathfrak{f}(\tau))$, where $\mathfrak{f}$ is a modular function of some level and $\tau \in \mathcal{O}_K$. However, as worded best by Kontogeorgis in \cite{konto}: \textit{So far it seems that all known class invariants were found out of luck or by extremely ingenious people like Ramanujan}.
\end{rmk}

\section{An efficient CM-construction} \label{efficient}
The following is a discussion of an algorithm described in \cite{broker} and \cite{broker2}. The algorithm reduces the time in calculating the Hilbert class polynomial by minimizing $|D|$, which can be done by constructing $D$ from a set of basis of primes.

Let $N$ be a rational prime. Recall that in constructing an elliptic curve of order $N$, a bottleneck is to construct the Hilbert class polynomial $H_D(X)$, where $D = (p+1-N)^2-4p$ and $p$ belongs in the Hasse's interval $H_N$. Even though it is best to use the field discriminant of $\mathbb{Q}(\sqrt{D})$, which is essentially the square-free part of $D$, the time complexity is too large for practical purpose as there does not exist a known polynomial time algorithm in computing the square-free part of an integer.

Instead of using a top-down approach as above in finding $D$, we can construct $D$ from a set of basis of primes. Note that given a prime $p \in H_N$, we have the discriminant
\begin{equation} \label{rewrite}
(p+1-N)^2 - 4p = (N+1-p)^2-4N = k^2D,
\end{equation}
for some fundamental discriminant $D$. It follows that for a fundamental discriminant $D$, if we can find a solution to the equation
\begin{equation}
x^2+y^2|D| = 4N,
\end{equation}
for some positive integers $x,y$ with $p=N+1\pm x$ prime, then we can construct an elliptic curve $E/\mathbb{F}_p$ of order $N$. Here we are using the symmetry $N \in H_p$ if and only if $p \in H_N$. Hence, we are trying to find a discriminant $D$ for which both $H_D(X) \pmod{p}$ and $H_D(X) \pmod{N}$ split completely.

From equation \ref{rewrite}, we note that for any odd prime $\ell \mid D$, we have 
\begin{equation} \label{goodprimes}
1 = \legendre{N}{\ell} = \legendre{(-1)^{\frac{\ell-1}{2}}\ell}{N},
\end{equation}
where the second equality comes from the Law of Quadratic Reciprocity. Letting $\ell^* = (-1)^{(p-1)/2}l$, we find that $D$ consists of primes $\ell$ for which $\ell^*$ is a quadratic residue modulo $N$. Moreover, since $N$ is assumed to be a rational prime, we must have $D \equiv 5 \pmod{8}$. Hence, we see that $D$ is a product of the primes $\ell$ satisfying Equation \ref{goodprimes}.

\begin{rmk} The fact above regarding the primes $\ell$ dividing $D$ can be easily seen using Class Field Theory. Recall that for a rational prime $N$, the equation $4N = x^2+y^2|D|$ has a solution in $\mathbb{Z}^2$ exactly when $N$ splits completely in the Hilbert class field $K_D$ of the quadratic field $K = \mathbb{Q}(\sqrt{D})$. In particular, $N$ splits completely in any subfield of $K_D$ of $K$. Hence, $N$ must splits completely in the genus field $G_D$ (see \cite[Theorem 6.1]{cox}) and all of its quadratic subfields $\mathbb{Q}(\sqrt{\ell^*})$ for primes $\ell \mid D$, which happens exactly when $\ell^*$ is a square modulo $N$.
\end{rmk}

The Prime Number Theorem states that
\begin{equation}
\lim_{x\rightarrow \infty} \frac{\pi(x)}{x/\log(x)} = 1,
\end{equation}
where $\pi(x)$ is the number of rational primes less than or equals to $x$. It follows that for integer $N$ sufficiently large, we expect $1$ out $\log(N)$ to be prime. Hence, instead of searching for all suitable primes $\ell$ in the interval $[1,N]$ at once, we search for $\ell$ in one sub-interval at a time starting with $[1,\log(N)]$ and ends with $[(m-1)\log(N)+1,N]$, where $m = \lfloor N/\log(N) \rfloor$. However, such a partition of the interval $[1,N]$ is only for ease of computing time complexity; in practice, we work with one prime $\ell$ at a time, which is better for Fibonacci primes due to Conjecture 4 of \cite{fibform}. Furthermore, It is noted in \cite{broker} that it is enough to consider $D$ comprising of at most two odd primes, where each prime $\ell < 2 \log(N)$.

In the case of a Fibonacci prime $f_q$, note that $\legendre{\ell^*}{f_q} = \legendre{\ell}{f_q}$ due to the fact that $f_q \equiv 1 \pmod{4}$, by Lemma \ref{mod4}. Let $P_q$ be the list of all primes $\ell < 2\log(f_q)$ such that $\legendre{\ell}{f_q} = 1$. We will describe an algorithm to construct a list $S_q$ of good discriminants from $P_D$. Here, we say that $D$ is a good discriminant if $D$ is of the form $D = -\ell_0$ or $D = - \ell_1 \ell_2$, and $D \equiv 5 \pmod{8}$, where $\ell_0,\ell_1,\ell_2 \in P_q$.

\begin{alg} \label{goodD} Suppose that we are given a list $P_q$ of odd primes $\ell < 2 \log(f_q)$ such that $\legendre{\ell}{f_q} = 1$. Let $N$ be the cardinality of $P_q$ and assume that the primes $\ell$ are listed in increasing order. This algorithm create the list $S_q$ of good discriminants.
\begin{enumerate}[nolistsep] 
\item Let $S_q$ be an empty list.
\item Let $k=0$.
\item If $k=N$, end the algorithm, else let $D = -P_q[k]$.
\item If $D = -P_q[k] \equiv 5 \pmod{8}$, append $D$ to $S_q$.
\item For $m=0,\ldots,k-1$, if $D = -P_q[m]\cdot P_q[k] \equiv 5 \pmod{8}$, append $D$ to $S_q$.
\item Increase $k$ by $1$ and return to Step 3.
\end{enumerate}
\end{alg}

This efficient construction of $D$ provides a degree of control of the class number. For security reason, we do not want $D$ to be too small as ECC can be attacked via using an isogenous curve. To ensure that $D$ is not too small, we could use the well known fact that the class number of $C(D)$ is approximately $\sqrt{-D}$ by Brauer-Siegel Theorem. Furthermore, the work of Goldfeld, Gross, Zagier and Osterle in the 1980s provide an easily  computable lower bound (see \cite{zag} and \cite[pp. 135]{cox}):
\begin{equation}
h(D) > \frac{1}{K} \log(-D)\prod_{p \mid D}^*\left(1-\frac{2\sqrt{p}}{p+1}\right),
\end{equation}
where $K = 55$ if $\gcd(D,5077)=1$ and $K=7000$ otherwise, and the product is taken over all prime divisors $p$ of $D$ except the largest prime.

Now we will approximate a lower bound for the expected number of such $D$ following \cite{broker2}, but we will provide a bit more details.  Recall that in the Algorithm \ref{mainalg}, we loop through the good discriminants $D \in S_q$ until we can find one such that $4f_q = x^2 + y^2|D|$ and $f_q + 1 \pm x$ is a prime for some positive integers $x,y$. As mentioned in Theorem \ref{complexity}, we should expect to find $O(\log(f_q))$ many $D$ by the Chebotarev Density Theorem, which follows from Lemma \ref{enough} with the bound $B = O(\log^2(f_q))$

\begin{thm}
Let $K_1,\ldots,K_n$ be distinct imaginary quadratic fields of odd class number. Let $H_1,\ldots,H_n$ be ring class fields of $K_1,\ldots,K_n$ respectively, such that $[H_i:K_i]=n_i$ are all odd. Let $H=\prod_{i=2}^n H_i$. Then $H_1 \cap H = \mathbb{Q}$.
\end{thm}
\begin{proof}
See \cite[Theorem 4.4]{oddclass}.
\end{proof}

\begin{lem}
The order of the class group $C(D)$ is odd exactly for discriminants $D$ of the form $D = -q$, where $q \equiv 3 \pmod{4}$.
\end{lem}
\begin{proof}
Recall that Genus Theory states that the number of elements of order $2$ in $C(D)$ is $2^{t-1}-1$, where $t$ is the number of odd prime divisors of $D$. The result follows. See also the argument preceding Proposition \ref{parity}.
\end{proof}

\begin{cor} \label{lind}
Let $q_1,q_2, \ldots, q_n$ be a set of primes such that $q_i \equiv 3 \pmod{4}$ for $i=1,\ldots,n$. Let $H_{q_1},\ldots, H_{q_n}$ be the Hilbert class fields of the imaginary quadratic fields $\mathbb{Q}(\sqrt{-q_1}),\ldots, \mathbb{Q}(\sqrt{-q_n})$, respectively. Let $H = \prod_{k=2}^n H_{q_k}$. Then $H_{q_1} \cap H = \mathbb{Q}$.
\end{cor}

Corollary \ref{lind} can be proven quickly by looking at the ramified primes, as noticed in \cite{broker}. Indeed, note that $q_k$ is the only rational prime that ramified in $H_{q_k}$, and $q_k$ does not ramify in $\prod_{m \neq k} H_{q_m}$. Hence the intersection must be $\mathbb{Q}$. It follows that the class fields $H_{q_k}$ are linearly independent over $\mathbb{Q}$.

\begin{lem} \label{enough}
 Let $N$ be a rational prime and let $P(B) = \{\ell : \ell \text{ is prime and } \ell \leq B\}$. The number $S(B)$ of primes $\ell \in S(B)$ such that $4N = x^2+\ell y^2$ for some integers $x,y$ is approximately $\sqrt{B}/\log(B)$. 
\end{lem}
\begin{proof}
Chebotarev Density tells us that $P(B)$ is of size $O(B/(2\log(B))$. Recall that $4N = x^2 + \ell y^2$ for some integers $x,y$  if and only if $N$ splits completely in the Hilbert class field $H_{\ell}$ of $K_{\ell} = \mathbb{Q}(\sqrt{-\ell})$ if and only if $N$ is the norm of some principal element of $\mathcal{O}_{\ell}$, the ring of integers of $K_{\ell}$. By the Chebotarev Density Theorem, the density of rational primes that split in $H_{\ell}$ is $1/2h_{\ell}$, where $h_{\ell}$ is the class number of  $\mathcal{O}_{\ell}$, the ring of integers of $K_{\ell}$. If $N$ splits  completely in $H_{\ell}$, then there exist two elements in $\alpha, \beta \in \mathcal{O}_{\ell}$ such that their norm is equal to $N$. It follows that the expected number of elements in $\mathcal{O}_{\ell}$  for which $N$ is the norm of is $1/h_{\ell} \approx 1/\sqrt{\ell}$. Hence, since the class fields $H_{\ell}$ are linearly disjoint over $\mathbb{Q}$, the expected number of primes $\ell$ for which $N$ split completely in $H_{\ell}$ is approximately
\begin{equation}
\sum_{\ell \in P(B)} \frac{1}{\sqrt{\ell}} \approx \left(\frac{B}{\log(B)}\right)\frac{1}{\sqrt{B}} = \frac{\sqrt{B}}{\log(B)}.
\end{equation}
\end{proof}

\section{Isogeny volcanoes} \label{volcano}
To compute the Hilbert class polynomial $H_D(X) \pmod{p}$, it is best avoid the complex-analytic method when $|D| > 10^{10}$, as this is the practical upper limit due to storage size (see \cite{as3}). In general, all methods in computing $H_D(X) \pmod{p}$ has time complexity $\widetilde{O}(|D|)$, with the difference being the storage required for each method. Here the discriminants in Algorithm \ref{mainalg} is of size $O(\log^2(f_q))$, so the time complexity is at most $\widetilde{O}(\log^2(f_q))$ each time the CRT method is called. Now in the case that $C(D)$ is composite such as the case when $D = -\ell_1 \ell_2$, a root of $H_D(X) \pmod{q}$ can be obtained directly without even knowing their coefficients (see \cite{as3}). As each discriminant $D$ in Algorithm \ref{mainalg} is either of the form $D = -\ell_0$ or $D = -\ell_1 \ell_2$, the class group $C(D)$ is highly cyclic, so we may only have walk around the surface of one $\ell$-volcano. Furthermore, each such discriminant $D$ is fundamental so we do not have to worry about explicitly computing the endormorphism ring of elliptic curves as outlined in Algorithm 1.2 of \cite{as2}. It is expected that the CRT method to be faster in our scenario.

We will now provide a quick overview of the method of computing $H_D(X) \pmod{q}$ using Chinese Remainder Theorem (CRT) following \cite{as2}. By the Chebotarev Density Theorem, given a negative discriminant $D$ the set of primes
\begin{equation}
\mathcal{P}_D=\{\eta > 3 \text{ prime }: 4\eta = t_{\eta}^2 + v_{\eta}^2|D| \text{ for some } t_{\eta},v _{\eta}\in \mathbb{Z}^+\}
\end{equation}
is infinite and of density $1/2h_D$, where $h_D$ is the order of the class group $C(D)$. Suppose that we wish to compute $H_D(X) \pmod{p}$ for some (very large) prime $p$. The CRT method is to compute $H_D(X) \pmod{\eta}$ for an optimized finite set $S(D)$ of primes $\eta \in \mathcal{P}$ so that 
\begin{equation}
\prod_{\eta \in S(D)} \eta > 2B,
\end{equation}
where $B$ is an upper bound for the coefficients of $H_D(X)$. By the CRT, we can explicitly determine $H_D(X) \pmod{p}$. Now we will describe how to find $H_D(X) \pmod{\eta}$ for $\eta \in S(D)$.

Let $K$ be the imaginary quadratic field $K = \mathbb{Q}(\sqrt{D})$, and suppose $4\eta = t^2 + v^2|D|$ for some positive integers $t,v$. By Complex Multiplication, each root $r$ of $H_D(X) \pmod{\eta}$ corresponds to an isomorphism class of elliptic curves with endomorphism ring isomorphic to $\mathcal{O}_K$. Furthermore, each curve has trace $\pm t$. 

  Let Ell$_t(\mathbb{F}_{\eta})$ be the set of the j-invariants of all elliptic curves over $\mathbb{F}_{\eta}$ with trace equals to $t$. Hence, each element $j_0$ of Ell$_t(\mathbb{F}_{\eta})$ represents the class of $E/\mathbb{F}_{\eta}$ and its twists, where $E$ is a curve with j-invariant $j_0$. Let Ell$_{\mathcal{O}}(\mathbb{F}_{\eta})$ be the set of all roots of $H_D(X) \pmod{p}$, which correspond to the j-invariant of all elliptic curves over $\mathbb{F}_{\eta}$ with endomorphism ring equals to $\mathcal{O}$. We have the following set inclusions
\begin{equation}
\text{Ell}_{\mathcal{O}}(\mathbb{F}_{\eta}) \subset \text{Ell}_t(\mathbb{F}_{\eta}) \subset \mathbb{F}_{\eta}.
\end{equation}

A key observation is that the set Ell$_t(\mathbb{F}_{\eta})$ consists of isogenous curves as they have trace $t$ over the same finite field. Hence, given $j(E) \in$ Ell$_t(\mathbb{F}_{\eta})$, it is discovered that there exists an efficient method in obtaining an isogenous curve $E'$ such that $j(E') \in $ Ell$_{\mathcal{O}}(\mathbb{F}_{\eta})$, the foundation of which is based on Kohel's thesis \cite{kohel}.

Kohel's work describes a method to explicitly compute the endomorphism ring of an ordinary elliptic curve $E$ over finite field, which is isomorphic to an order $\mathcal{O}$ of some imaginary quadratic field $K$. We have the following containment:
\begin{equation}
\mathbb{Z}[\pi_E] \subset \mathcal{O} \subset \mathcal{O}_K.
\end{equation} Let $u = [\mathcal{O}_K : \mathcal{O}]$ and let $v = [\mathcal{O} : \mathbb{Z}[\pi_E]]$. The index $w = [\mathcal{O}_K : \mathbb{Z}[\pi_E]]$ is equal to $uv$. Let $\nu_{\ell}$ be the standard $\ell$-adic valuation. Kohel's work \cite{kohel} shows that computing the endomorphism ring of $E$ is equivalent to known the $\nu_{\ell}(w)$ for various primes $\ell$ (See \cite[Proposition 2]{as2}).

Let $\ell \neq \eta$ be a prime, and let $\Gamma_{\ell,t}(\mathbb{F}_{\eta})$ be the undirected graph with $V = $ Ell$_t(\mathbb{F}_{\eta})$ as vertices. There is an edge between $j(E),j(E') \in V$ exactly when $\varphi_{\ell}(j(E),j(E')) = 0$, where $\varphi_{\ell}(x,y)$ is the well-known classical modular polynomial (see \cite[ch. 11]{cox}). The equation $\varphi_{\ell}(j(E),j(E')) = 0$ is satisfied exactly when there is an isogeny of degree $\ell$ between $E$ and $E'$. These modular polynomials $\Phi_{\ell}(X,Y)$ are precomputed in Step \ref{smodular} since Algorithm \ref{crt} are to be called for $O(\log(f_q))$ many discriminants, and they will be reused later in Step \ref{selkies} and Step \ref{seigenvalue} of Algorithm \ref{mainalg}.

With at most two exceptions, the components of $\Gamma_{\ell,t}(\mathbb{F}_{\eta})$ are $\ell$-volcanoes (see \cite{as2} for definitions), but since we have excluded $j=0,1728$ in Algorithm \ref{mainalg} those exceptions do not occur. The graph resembles a volcano as can be seen in Figure \ref{volfig}. Each $\ell$-volcano can be partitioned into levels $V_0,\ldots, V_d$, where each level of the volcano represents elliptic curves with the same endomorphism ring, and the depth each $\ell$-volcano is $d = \nu_{\ell}(w)$. The bottom (floor) of the volcano contain curves with endomorphism ring generated by the Frobenius automorphism, while the top (surface) contains curves with the full ring of integers $\mathcal{O}_K$ as their endomorphism ring, where $K = \mathbb{Q}(\sqrt{D})$.

\begin{figure}[h]
\centering
\includegraphics[scale=.5]{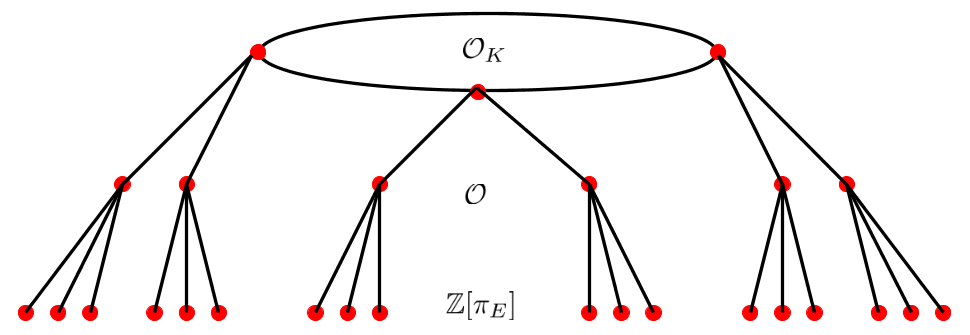}
\caption{A $3$-volcano of depth $2$, with a $3$-cycle on the surface.}
\label{volfig}
\end{figure}

The idea is to start with a random curve with trace $t$, which is actually very difficult to find. As suggested by Sutherland, it is best to use the idea of picking points from modular curves $X_1(m)$ (see Section \ref{modular}). Once we have a $j_0 \in $ Ell$_t(\mathbb{F}_{\eta})$, we go to the $\ell$-volcano that contains $j_0$ and we replace $j_0$ with the j-invariant at the level $\nu_{\ell}(w)$. If we perform this for each $\ell \mid w$, then by Proposition 2 of \cite{as2} the final $j_0 \in$  Ell$_{\mathcal{O}}(\mathbb{F}_{\eta})$, as desired.  We then choose variety of primes $\ell \neq p$ so that $\ell \nmid u$, in which case $j_0$ is on the top of each $\ell$-volcano. Finally, we use the action of $C(D)$ on $j_0$ to obtain all the other elements of Ell$_{\mathcal{O}}(\mathbb{F}_{\eta})$ by walking on the surface of these $\ell$-volcanoes.

The action of $C(D)$ on Ell$_{\mathcal{O}}(\mathbb{F}_{\eta})$ is a very interesting aspect of the algorithm. Let $\ell \neq p$ be a prime such that $\legendre{D}{\ell} \neq -1$, which we will classify as an Elkies prime in Section \ref{schoof}. There exists a prime ideal $\mathfrak{a}$ of $\mathcal{O}$ such that $(\ell) = \mathfrak{a} \overline{\mathfrak{a}}$, that is, $\ell$ is the norm of $\mathfrak{a}$. There is a \textit{prime form} $(\ell,b_{\ell},c_{\ell})$ of discriminant $D$ (see Section \ref{generators})  corresponding to the ideal $\mathfrak{a}$ such that $\ell$ is the norm of the class $[(\ell,b_{\ell},c_{\ell})]$ of $C(D)$ represented by $(\ell,b_{\ell},c_{\ell})$. The order ord$_D(\ell)$ of $[(\ell,b_{\ell},c_{\ell})]$ in $C(D)$ is equals to the number of elements of Ell$_{\mathcal{O}}(\mathbb{F}_{\eta})$ that lie on the surface of the $\ell$-volcano (see \cite[Proposition 3]{as2}).

Suppose $j_0 \in$ Ell$_{\mathcal{O}}(\mathbb{F}_{\eta})$ lies on the surface $V_0$ of an $\ell$-volcano. One walks a path of length $d$ to obtain a list of elements $[j_0,j_1,\ldots,j_d]$ of Ell$_t(\mathbb{F}_{\eta})$. If the level of $j_d \not \in V_d$, then $j_2 \in $  Ell$_{\mathcal{O}}(\mathbb{F}_{\eta})$, else we try another path of length $d$. We perform this walking on each subsequent found $j \in $  Ell$_{\mathcal{O}}(\mathbb{F}_{\eta})$ until the number of elements found is equal to ord$_D(\ell)$. In general, to obtain elements on the surface, we use primes $\ell$ so that the depth $d$ is $0$, that is, $\ell \nmid v$. We use primes $\ell \mid v$ only when it is easier to find roots of $\Phi_{\ell}(X,j(E))$. Hence, we observe that most of the orbit computations are done on $\ell$-volcanoes of depth $0$.

Note that when the depth $d =0$, walking around the surface $V_0$ is relatively straightforward from Proposition 6.2 of \cite{schoof} regarding Elkies prime. If $\legendre{D}{\ell}=0$, then ord$_D(\ell) = 2$. The polynomial $\Phi_{\ell}(X,j_0)$ has exactly one other root $j_1$ that lies on the surface $V_0$. Hence, the surface $V_0$ in this case is just a line segment. Now when $\legendre{D}{\ell} = 1$, the polynomial $\Phi_{\ell}(X,j_{i-1})/(X-j_i)$ has exactly one root $j_{i+1}$. Of course, here the surface $V_0$ is a cycle of length ord$_{D}(\ell)$.

There are many considerations must be taken for this method to be efficient. The primes $\eta$ in $S(D)$ must be chosen carefully so that the density of curves with $t$ is high, that is, we choose primes $\eta$ that enlarge Ell$_{\mathcal{O}}(\mathbb{F}_{\eta})$ while keeping the size of Ell$_{\mathcal{O}}(\mathbb{F}_{\eta})$ the same. We also want primes $\eta$ so that the index $v=[\mathcal{O}:\mathbb{Z}[\pi_E]]$ has small prime divisors as the depth $d$ of each $\ell$-volcano depends entirely on $v$ when $D$ is fundamental. The primes $\ell$ must be chosen so that movements on the $\ell$-volcanoes are easiest. The action of $C(D)$ on Ell$_{\mathcal{O}}(\mathbb{F}_{\eta})$ is easier to compute if we use the prime forms of $C(D)$ (see Section \ref{generators}). Finally, the computations for the CRT must be updated continuously once $H_D(X) \pmod{\eta}$ is obtained for each $\eta \in S(D)$.

\begin{alg} Let $D$ be a fundamental discriminant of size $O(\log^2(f_q))$ from Step \ref{sdisc} from Algorithm \ref{mainalg}. This algorithm finds $H_D(X) \pmod{p}$ using Chinese Remainder Theorem. \label{crt}
\begin{enumerate} [nolistsep]
\item Let $\mathcal{P}_D = \{\eta > 3 \text{ prime} : 4\eta=t_\eta^2+v_\eta^2|D|$ \text{for some} $t_\eta,v_\eta \in \mathbb{Z}\}$.
\item Choose an optimized list $S(D)$ of primes from $\mathcal{P}_D$.
\item Let $k = 0$.
\item Let $\eta = S[k]$.
\item Find a curve $E$ with $j(E) \in $ Ell$_t(\mathbb{F}_\eta)$ following Section \ref{modular}.
\item Use the precomputed modular polynomials $\Phi_{\ell}(X,Y)$ to find an isogenous $E'$ such that $j_0=j(E') \in $ Ell$_{\mathcal{O}}(\mathbb{F}_\eta)$.
\item Compute the prime forms for discriminant $D$ using Algorithm \ref{generators}.
\item Use the precomputed modular polynomials $\Phi_{\ell}(X,Y)$ and prime forms to compute the orbit of the group action of $C(D)$ on $j_0$ to obtain all the elements of Ell$_{\mathcal{O}}(\mathbb{F}_\eta)$.
\item Compute $H_D(X) \pmod{\eta}$ by expanding
\begin{equation}
H_D(X) \pmod{p} = \prod_{j \in \text{Ell}_{\mathcal{O}}(\mathbb{F}_\eta)}  (X - j) \pmod{\eta}.
\end{equation}
\item Increase $k$ by $1$ and return to Step 4 if $k < |C(D)|$, else go to the next step.
\item Use CRT to compute $H_D(X) \pmod{p}$.
\end{enumerate}
\end{alg}

Note that precomputing $\Phi_{\ell}(X,Y)$ is negligible as each prime $\ell < 6\log^2(4\log^2(f_q))$, and the time complexity for computing each $\Phi_{\ell}(X,Y)$ is $\widetilde{O}(\log^3(\ell))$. 

\section{Sampling from modular curve} \label{modular}
Recall that if we wish to construct an elliptic curve $E/\mathbb{F}_p$ of order $N$, then we need to find one with trace $t = p+1-N$. The naive method in finding a curve $E$ of trace $t$ is to look for $E: Y^2=X^3+aX-a$, where $1 \leq a \leq p-1$, such that 
\begin{equation} \label{trace}
(p+1)(1,1) = \pm t(1,1),
\end{equation}
as the point $(1,1)$ on $E$ (though $E$ may not have trace $t$). Then we compute the order of $E$ and find its twist if necessary. However, this naive method requires the computation of the order of around $2\sqrt{p}$ curves. To accelerate the search we need to reduce our sample size, and in \cite{as} Sutherland does this by searching for points on on the modular curve $X_1(d)$ for various $d \mid N$.

Recall that by Mazur's theorem (see \cite[Theorem 7.5]{silverman1}), if $E$ is an elliptic curve over $\mathbb{Q}$, then the order of a non-trivial torsion point $P$ of $E$ is a number in the set $\mathcal{T}=\{2,3,4,5,6,7,8,9,10,12\}$. Let $p$ be an odd prime that does not divides the discriminant of $E$, that is, $p$ is a prime of good reduction, then the reduction map
\begin{equation}
E_{\text{torsion}} \rightarrow \overline{E}(\mathbb{F}_p)
\end{equation}
is injective (see page 123 of \cite{silverman3}). This provides a way to obtain a curve over $\mathbb{F}_p$ with order divisible by $d \in \mathcal{T}$. However, the possible orders of $\overline{E}$ is very limited if $E$ is defined over $\mathbb{Q}$. We could try to find a curve $E$ over some finite extension $K/\mathbb{Q}$; however, its reduction may not have coefficients in $\mathbb{F}_p$. For example, the reduction of $E/\mathbb{Q}(d)$ has coefficients in $\mathbb{F}_p$ exactly when $d \pmod{p}$ is a square. As Sutherland suggests in \cite{as}, we should look for points on the curve $Y_1(d)/\mathbb{F}_p$, where $Y_1(d)$ is the affine subcurve of $X_1(d)$.

The modular curve $X_1(m)$ classifies pairs $(E,P)$ of elliptic curves $E$ with a fixed point $P \in E(\mathbb{C})$ of order $m$ up to isomorphism over $\mathbb{C}$ (see \cite{alr}). To find an elliptic curve of order $m$, we narrow our search to within $Y_1(d)$ for some $d \mid m$. We want $d$ to be reasonably small due to the cost in finding such points. Furthermore, if $d_1 < d_2$ are small divisors of $m$, we generate several curves from $X_1(d_1)$ and test them for $d_2$-torsion.

To find a curve $E/\mathbb{F}_p$ of trace $t$, we search for points on $Y_1(d)/\mathbb{F}_p$, where $d \mid p+1-t$ or $d \mid p+1+t$, preferably both. From the given point, we can find its Weierstrass equation, and if the curve is singular, we find another point. In \cite{as}, Sutherland has optimized the search for points on $Y_1(d)/\mathbb{F}_p$ by computing the \textit{raw form} $F_d(r,s)$ of $X_1(d)$. Sutherland and Hoeij has computed $F_d(r,s)$ for $d$ up to $100$. It is recommended in \cite{as} to use such forms for only $d$ up to $40$ due to the cost of finding points of $F_d(r,s) =0$.

In theory, one could directly apply this searching method to construct elliptic curves of prime order $N$ over some prime field $\mathbb{F}_p$, provided that $|t| = |p+1-N| \leq 2\sqrt{p}$. We will describe this algorithm below; however, it is impractical for large prime $p$.

To find an elliptic curve over $\mathbb{Z}/p\mathbb{Z}$ of trace $t=p+1-N$, we find a small divisor $d$ of $N_0=p+1+t$ and search within $Y_1(d)/(\mathbb{Z}/p\mathbb{Z})$. Once we find a curve of trace $t$, we compute its twist. Of course, this method fails when $N_0$ is not smooth.

For each of the probable prime $p$ found from Step \ref{ssplit} and verified in Step \ref{srm}, we can apply the following algorithm to construct an elliptic curve $E/(\mathbb{Z}/p\mathbb{Z})$ of order $f_q$. Note by construction $|t|=|p+1-f_q| \leq 2\sqrt{p}$ is automatically satisfied.

\begin{alg} \label{modularalg} Algorithm to find an elliptic curve $E/(\mathbb{Z}/p\mathbb{Z})$ of trace $t=p+1-f_q$.
\begin{enumerate}[nolistsep]
\item Compute $N_0 = 2(p+1)-f_q$.
\item Find a small divisor $d$ of $N_0$.
\item \label{findpoints} Search for points on $Y_1(d)/(\mathbb{Z}/p\mathbb{Z})$ using the optimized planar equation $F_d(r,s)=0 \pmod{p}$.
\item Find a Weierstrass equation for $E/(\mathbb{Z}/p\mathbb{Z})$.
\item If $E$ is singular, return to Step \ref{findpoints} and find another point, else go to the next step.
\item Find a twist of $E$ with order $f_q$ if necessary.
\end{enumerate}
\end{alg}

\section{Prime forms of class group} \label{classgroup}
Recall that for a negative discriminant $D$, the elements of $C(D)$ consists of equivalence classes of binary quadratic forms $q(X,Y) = aX^2+bXY+cY^2$ such that $D = b^2-4ac$. Moreover, each class is represented by exactly one reduced form. We wish to find generators for $C(D)$ where $a$ is a rational prime.

Let $\ell$ be a rational prime. Note that if $D = b^2 - 4\ell c$, then $D$ is a quadratic residue modulo $4\ell$. From this observation, we can easily obtain quadratic forms of $C(D)$ with $a$ prime (see \cite{bv} and page 251 of \cite{cohen1}). Indeed, assume $D$ is a quadratic residue modulo $\ell$, and write $b_{\ell}$ for its square root modulo $\ell$. By taking $\ell-b_{\ell}$ if necessary, we may assume that $b_{\ell}$ is a square root of $D$ modulo $4\ell$. Then the form $(\ell,b_{\ell},(b_{\ell}^2-D)/4\ell)$ has discriminant $D$. We call such a form to be a \textit{prime form} of $C(D)$. In fact, every form $F$ of $C(D)$ is a product of such forms:
\begin{lem} (Lemma 5.5.1 of \cite{cohen1}) Let $(a,b,c)$ be a primitive positive definite quadratic form of discriminant $D < 0$, and $a = \prod_{\ell} {\ell}^{\nu_{\ell}}$ be the prime factorization of $a$. Then we have up to equivalence:
\begin{equation}
(a,b,c) = \prod_{\ell} F_{\ell}^{\epsilon_{\ell} \nu_{\ell}},
\end{equation}
where $F_{\ell}$ is a prime form corresponding to $\ell$, and $\epsilon_{\ell} = \pm 1$ is defined by the congruence
\begin{equation}
b \equiv \epsilon_{\ell} b_{\ell} \pmod{2 \ell}.
\end{equation}
\end{lem}

  Assuming the Extended Riemann Hypothesis (ERH), restricting $\ell$ to be primes $\ell \leq 6\log^2(|D|)$ yields a sequence of generators for $C(D)$ (see \cite{bach}). It is suggested in \cite{cohen1} that in practice it is better to search for primes $\ell \leq B(D)$, where 
\begin{equation} \label{bound}
B(D) = \max\left(6(\log(|D|))^2,L(|D|)^{1/\sqrt{8}}\right)
\end{equation}
and
\begin{equation}
L(x) = e^{\sqrt{\log(x) \log(\log(x))}}.
\end{equation}

\begin{alg} \label{generators} Algorithm to find prime forms of $C(D)$
\begin{enumerate}[nolistsep]
\item Let $R$ and $\mathcal{F}$ be empty sets.
\item For each prime odd prime $\ell \leq B(D)$ such that $\legendre{D}{\ell} = 1$, find a square root $b_{\ell}$ of $D$ modulo $\ell$. Take $b_{\ell} = \ell-b_{\ell}$ if necessary, we may assume that $b_{\ell}^2 \equiv D \pmod{4\ell}$, and we store the pair $(b_{\ell},\ell)$ into the set $R$. 
\item For each pair $(b_{\ell},\ell) \in R$, store the form $(a_{\ell},b_{\ell},c_{\ell})$ into $\mathcal{F}$, where $c_{\ell} = (b_{\ell}^2-D)/(4\ell)$.
\item Return $\mathcal{F}$.
\end{enumerate}
\end{alg}

As mentioned in Section \ref{volcano}, the prime forms for $C(D)$ are used in \cite{as2} to create a polycyclic presentation of $C(D)$. These forms provide efficient walks on isogeny volcanoes, and the points on the surface of the volcanoes provide roots of $H_D(X) \pmod{q}$ for some prime $q$. As suggested by Sutherland in \cite{as2}, to make sure these prime forms generate $C(D)$ unconditionally, one computes the order of $C(D)$ and add more forms if necessary; however, this is not practical for large discriminants such as the ones in our scenario. 

\section{Computing square root and Cornacchia's algorithm} \label{roots}
There are various algorithms to find a square root of an integer $a$ modulo prime $p$, assuming that there exists one. In the easiest case $p \equiv 3 \pmod{4}$, the square root of $a$ modulo $p$ is given by $x = a^{(p+1)/4)} \pmod{p}$. The remaining cases are $p \equiv 5 \pmod{8}$ and $p \equiv 1 \pmod{8}$. For the case $p \equiv 1 \pmod{8}$, we will use Tonelli-Shanks algorithm, which is useful when computing multiple square roots. Finally, the case $p \equiv 5 \pmod{8}$ is relatively straightforward. If $p \equiv 5 \pmod{8}$ and $a$ is a square modulo $p$, then
\begin{equation}
\sqrt{a} \pmod{p} = \left\{
        \begin{array}{rl}
            a^{(p+3)/8} \pmod{p}, &  \text{ if } a^{(p-1)/4} \equiv 1 \pmod{p} \\
            2a(4a)^{(p-5)/8} \pmod{p}, & \text{ if } a^{(p-1)/4} \equiv - 1 \pmod{8}.
        \end{array}
    \right.
\end{equation}

We will now described Tonelli-Shanks algorithm, following \cite{cohen1}. Suppose we wish to compute the square root of $a$ modulo $p$. Here, we are assuming that $\legendre{a}{p} = 1$. Write
\begin{equation}
p - 1 = 2^e m,
\end{equation}
where $m$ is odd. The multiplicative group $\left(\mathbb{Z}/p\mathbb{Z}\right)^*$ is cyclic of order $p-1$, and so its $2$-Sylow subgroup $G$ is cyclic of order $2^e$. Find an integer $n$ such that $\legendre{n}{p}=-1$, and compute $g=n^m$, which clearly belong to $G$. If $g^{2^{e-1}} = 1 \pmod{p}$, then $p$ is composite, else $g$ generates the 2-Sylow subgroup $G$. Similarly, the element $a^m \in G$, so $a^m g^k = 1 \pmod{p}$ for some integer $k$. It follows that a square root of $a \pmod{p}$ is given by
\begin{equation}
\sqrt{a} \pmod{p} = a^{(m+1)/2}g^{k/2} \pmod{p}.
\end{equation}

It would be useful to compute the square roots of all the elements in the 2-Sylow if $e$ is small and multiple square roots are to be computed. This is especially useful for Algorithm \ref{mainalg}. Furthermore, one amazing fact is that for each of the probable Fibonacci prime $f_q$ from \cite{online}, the $2$-Sylow subgroup of $\left(\mathbb{Z}/f_q\mathbb{Z}\right)^{\times}$ is very small (see Observation \ref{2val}). Hence, computing square roots modulo $f_q$ is relatively simple in all cases.

Below is a variant of the Tonelli-Shanks algorithm that is tailored for our purpose. We will break the Tonelli-Shanks algorithm into two parts to suit our purpose. If either part fails, then $f_q$ is not a prime.

\begin{alg} \label{pre} (Precomputations) This algorithm compute a square root of each element of the 2-Sylow subgroup $G$ of $\left(\mathbb{Z}/f_q\mathbb{Z}\right)^{\times}$. Suppose we are given a quadratic non-residue $n$ modulo $f_q$.
\begin{enumerate}[nolistsep]
\item Let $R(2,f_q)$ be an empty list.
\item Factor $f_q-1$ as $f_q-1 = 2^e m$, where $m$ is odd.
\item Compute $g=n^d \pmod{f_q}$.
\item Let $G$ be the subgroup generated by $g$.
\item Compute a square root $r_i$ of each element $g^i$ of $G$ and append the ordered pair $(g^i, r_i)$ to the list $R(2,f_q)$.
\end{enumerate}
\end{alg}

Here we are applying the standard Tonelli-Shanks algorithm repeatedly in the final step of Algorithm \ref{pre}. Of course, for some of the elements we do not have to. Now we will describe Tonelli-Shanks algorithm if we are given $R(2,f_q)$.

\begin{alg} (Variant of Tonelli-Shanks algorithm) \label{ts} Given $R(2,f_q)$ and $\legendre{a}{f_q} = 1$, this algorithm computes a root of $a \pmod{f_q}$.
\begin{enumerate}[nolistsep]
\item Compute $x = a^{(m+1)/2} \pmod{f_q}$ and $y = a^m \pmod{f_q}$.
\item Look up $R(2,f_q)$ for a square root of $z$ of $y$.
\item Then $\sqrt{a} = \pm x/z \pmod{f_q}$.
\end{enumerate}
\end{alg}

It is noted in \cite{am} that if we fail to find a square root modulo $p$ using the above algorithms, then $p$ is composite. It is unlikely for a composite number to pass Algorithm \ref{ts}. See \cite{will2} for a combination of square roots with primality tests.

An efficient algorithm to obtain a square root is useful in solving a variety of Diophantine equations, namely the equation $X^2+dY^2 = m$. Cornacchia's algorithm (see Section 1.5.2 of \cite{cohen1} and \cite{schoof}) can provide the unique solution of positive integers to the equation $X^2+dY^2=m$ if there exists one. The algorithm is essentially an application of the Euclidean algorithm. We will provide an overview of the algorithm following \cite{schoof}, where its proof by Lenstra can also be seen. In \cite{schoof}, the Cornacchia's algorithm is used to compute the order of an elliptic curve $E/\mathbb{F}_p$ if its endomorphism ring is known, and he also reversed this process to compute square roots modulo $p$ (see \cite{schoof2}).

If $\legendre{-d}{m}=-1$, then clearly the equation $X^2+dY^2=m$ does not have a solution. Assume $\legendre{-d}{m}=1$, and let $r_0$ be a root of $-d \pmod{m}$. Replacing $r_0$ with $m - r_0$ if necessary, we may assume that $r_0 \leq m/2$. Let $r_{-1}=m$. Use the Euclidean algorithm to find a sequence of non-negative integers $r_1,r_2,\ldots,r_k$ such that $r_k < \sqrt{m}$ and
\begin{equation} 
r_j \equiv r_{j-2} \pmod{r_{j-1}}, \hspace{0.2in} \text{ for } 1 \leq j \leq k.
\end{equation}
Then a unique solution of positive integers exists exactly when the real number
\begin{equation}
\sqrt{\frac{m-r_k^2}{d}}
\end{equation}
is an integer, and the solution is given explicitly by
\begin{equation} 
(x,y) = \left(r_k,\sqrt{\frac{m-r_k^2}{d}}\right).
\end{equation}

\begin{alg} (Cornacchia's Algorithm) \label{corn} Algorithm to find a solution to $x^2+dy^2=m$ if there exists one.
\begin{enumerate}[nolistsep]
\item Find a square root $r_0$ of $-d \pmod{m}$.
\item Replace $r_0$ with $m - r_0$ if necessary, we may assume $r_0 \leq m/2$.
\item Let $r_{-1} = m$.
\item Use the Euclidean algorithm to find a sequence of non-negative integers $r_1,r_2, \ldots, r_k$ satisfying  $r_k < \sqrt{m}$ and $r_j \equiv r_{j-2} \pmod{r_{j-1}}$, for $1 \leq j \leq k$.
\item Compute $s=(m-r_k^2)/d$.
\item A solution exists exactly when $s$ is an integer squared, and the unique solution of positive integers is given by $(x,y)=(r_k,\sqrt{s})$.
\end{enumerate}
\end{alg}

Note that each square root computation $\sqrt{a} \pmod{f_q}$ is just applications of a small number of exponentiations, which has time complexity $\widetilde{O}(\log^2(f_q))$. The Cornacchia's Algorithm is just an application of Euclidean Algorithm, which has time complexity $\widetilde{O}(\log(f_q))$. Hence, the time complexity for all algorithms in this section has time complexity of at most $\widetilde{O}(\log^2(f_q))$.

\section{Schoof's algorithm} \label{schoof}
Let $E/\mathbb{F}_p$ be an ordinary elliptic curve over $\mathbb{F}_p$ of the form $Y^2 -f(X) = 0$ for some cubic polynomial $f(X)$, whose j-invariant is neither $0$ or $1728$. Recall that the order $N$ and trace $t$ of $E$ are related by $N = p+1-t$. Hence, computing the order $E$ is equivalent to computing the trace of $E$. Schoof's algorithm \cite{schoof2} can compute $t$ in time complexity $O(\log^8(p)$. Improvements by Atkin and Elkies \cite{schoof} have lead to the SEA algorithm, which has time complexity $\widetilde{O}(\log^4(p))$.

Schoof's idea is to compute $t \pmod{\ell}$ for finitely manly primes $\ell \neq p$ and compute $t$ using the Chinese Remainder Theorem. Since $|t| \leq 2 \sqrt{p}$, we want to choose the set $S(E)$ of primes $\ell < p$ so that
\begin{equation}
\prod_{\ell \in S(E)} \ell > 4 \sqrt{p},
\end{equation}
in order to be able to determine $t$ uniquely.

To compute $t \pmod{\ell}$, we use the characteristic polynomial 
\begin{equation} \label{char}
X^2-tX + p = 0
\end{equation}
of the Frobenius automorphism $\pi_E$ of $E$. Let $P = (x,y)$ be a point of the $\ell$-torsion subgroup $E[\ell]$ of $E$. From Equation \ref{char}, we have
\begin{equation}
(\pi_E^2-t\pi_E+p)P = 0,
\end{equation}
and explicitly in terms of $x$ and $y$, we have
\begin{equation}
(x^{p^2},y^{p^2}) - t(x^p,y^p) + p(x,y) = 0,
\end{equation}
which implies
\begin{equation} \label{schoofidea}
(x^{p^2},y^{p^2}) + p_{\ell}(x,y) = t_{\ell}(x^p,y^p),
\end{equation}
where $t_{\ell} \equiv t \pmod{\ell}$ and $p_{\ell} \equiv p \pmod{p}$. Here, the $0$ acts as both the point at infinity and the morphism induced by $0$. Its meaning can be understood by context. 

Schoof's idea to find $t \pmod{\ell}$ is to plug $t_{\ell} = 1,\ldots, \ell-1$ into Equation \ref{schoofidea} until the equation is satisfied. In practice, we actually only have to  look at $1 \leq t_{\ell} \leq (\ell-1)/2$ by looking at the second coordinate. However, instead of working with one $\ell$-torsion point at a time, we work with the entire $\ell$-torsion group $E[\ell]$ at once, that is, we perform computations in the ring
\begin{equation} \label{schoofeq}
R_{\ell} = \mathbb{F}_p[X,Y]/(\varphi_{\ell}(X),Y^2-f(X)),
\end{equation}
where $\varphi_{\ell}(X)$ is the $\ell$-division polynomial of $E$ (see \cite[pp. 105]{silverman1}). The degree of $\varphi_{\ell}(X)$ is of size $O(\ell^2)$, and each root of $\varphi_{\ell}(X)$ corresponds to the $X$-coordinate of some point in $E[\ell]$.

If we choose a prime $\ell$ of size $O(\log(f_q))$ such that $\legendre{t^2-4f_q}{\ell} \neq -1$, then we can compute $t_{\ell} \pmod{\ell}$ much quicker. The prime $\ell$ is called an Elkies prime. The reason for this drastic reduction in time complexity is due to the fact that one can use a polynomial $F_{\ell}(X)$ of degree of size $O(\ell)$ in Equation \ref{schoofeq}, instead of the polynomial $\varphi_{\ell}(X)$ with degree of size $O(\ell^2)$, though the trace $ t_{\ell} \pmod{\ell}$ is computed differently than as above. 

Note that since the trace $t$ is unknown, we can not immediately determine whether an integer is an Elkies prime. To determine whether $\ell$ is an Elkies prime, we such that fact that $\Phi_{\ell}(X,j(E))$ has a root exactly when $\ell$ is an Elkies prime (see \cite[Proposition 6.2]{schoof}). Here $\Phi_{\ell}(X,Y)$ is the classical modular polynomial that parametrizes pairs of $\ell$-isogenous elliptic curves over $\mathbb{C}$ (see \cite[ch. 11]{cox}). This interpretation carries over to finite fields of characteristic prime to $\ell$. The work \cite{as4} on the distribution of Atkin and Elkies primes shows that one should be able to find an Elkies prime quickly.

Since the polynomial $X^p-X$ splits completely over $\mathbb{F}_p$, the polynomial $\Phi_{\ell}(X,j(E)) \pmod{p}$ has a root $j_0$ in $\mathbb{F}_p$ exactly when
\begin{equation}
\gcd(\Phi_{\ell}(X,j(E)),X^p-X) > 1.
\end{equation}
As mentioned above, there exists an isogenous elliptic curve $E'$ such that $j(E') = j_0$. By Proposition 6.1 of \cite{schoof}, there exists a 1-dimensional subspace $C$ of $E[\ell]$ such that $E'$ and $E/C$ are $\mathbb{F}_p^{alg}$-isomorphic. Furthermore, the subspace $C$ is an eigenspace of $\pi_E$ for some eigenvalue $\lambda$, which corresponds to a root of $X^2-tX+p = 0$. Hence, if the eigenvalue $\lambda$ is known, the value $t_{\ell} = t \pmod{\ell}$ is easily computed.

Corresponding to the eigenspace $C$ of $\pi_E$, there exists a polynomial $F_{\ell}(X)$ of degree $(\ell-1)/2$, whose roots correspond to the distinct $X$-coordinates of the points in $C$. See \cite{schoof} for explicit calculation of the polynomial $F_{\ell}(X)$. Hence, to compute $\lambda$, we check which of the relations
\begin{equation}
\pi_E(X,Y) = (X^p,Y^p) = \lambda' \cdot (X,Y) \hspace{.2in} \lambda' = 1, \ldots, \ell-1 \pmod{F_{\ell}(X)}.
\end{equation} is satisfied.
The polynomials here that we are working with have degrees of size $O(\ell)$ instead of $O(\ell^2)$ as in Schoof's algorithm. The time complexity in computing $t_{\ell} \pmod{\ell}$ is drastically reduced.

Schoof's algorithm and the observations above provide us a way to remove unwanted curves from the list $\mathcal{E}_q$ that we obtain at the end of Step \ref{scheckorder} of Algorithm \ref{mainalg}. Each element of $(E,p,D)$ is an ordinary curve with possible order $f_q$ over the ring $\mathbb{Z}/p\mathbb{Z}$. The j-invariant of $E$ is not equal to $0$ or $1728$ and is a root of $H_D(X) \pmod{p}$ by construction. We remove $(E,p,D)$ from the list $\mathcal{E}_q$ if $\Phi_{\ell}(X,j(E))$ does not have a root for prime $\ell$ such that $\legendre{D}{\ell} \neq -1$. If $(E,p,D)$ does pass this test, we verify that a root $\lambda$ of $X^2-tX+p = 0$ is an eigenvalue of the Frobenius map $\pi_E$.

\begin{alg} (Elkies Primes Verification) \label{elkies} Let $\mathcal{E}_q$ be the list of curves obtained from Step \ref{scheckorder} of Algorithm \ref{mainalg}. Let $N$ be the cardinality of $\mathcal{E}_q$. 
\begin{enumerate}[nolistsep]
\item Let $k=0$ and let $E(D)$ be an empty list.
\item If $k = N$, then go to the final step, else let $(E,p,D)= \mathcal{E}_q[k]$.
\item Let $\ell = 2$.
\item Compute $r=\legendre{D}{\ell}$.
\item If $r \neq -1$, append $\ell$ to $E(D)$.
\item Compute $d=\gcd(\Phi_{\ell}(X,j(E)),X^p-X)$.
\item If $r \neq -1$ and $d = 1$ or $r = -1$ and $d > 1$, then remove $(E,p,D)$ from $\mathcal{E}_q$ and return to Step 2, else go to the next step.
\item Let $\ell$ be the next prime of $\ell$. If $\ell < 2\log^2(4\log^2(f_q))$, then return to Step 4, else increase $k$ by $1$ and return to Step 2.
\item Output $\mathcal{E}_q$.
\end{enumerate}
\end{alg}

\begin{alg} (Eigenvalue Verification) \label{eigenvalue} Let $\mathcal{E}_q$ be the list output by Algorithm \ref{elkies}. Let $N$ be the cardinality of $\mathcal{E}_q$.
\begin{enumerate}[nolistsep]
\item Let $k=0$.
\item If $k=N$, then go to the final step, else let $(E,p,D) = \mathcal{E}_q[k]$.
\item Let $\ell = E(D)[k]$.
\item Let $\lambda,\mu$ be the roots of $X^2-tX+p=0$.
\item Find polynomial $f_{\ell}$ given by
\begin{equation}
F_{\ell}(X) = \prod_{\pm P \in C}(X-P_x),
\end{equation}
as defined above.
\item Let $R_{\ell}$ be the ring defined by
\begin{equation}
R_{\ell} = \mathbb{F}_p[X,Y]/(F_{\ell}(X),Y^2-f(X)).
\end{equation}
\item Verify if either $(X^p,Y^p) = \lambda(X,Y)$ or $(X^p,Y^p) = \mu(X,Y)$ in the ring $R_{\ell}$. If either is satisfied, then increase $k$ by $1$ and return to Step 2. If neither is satisfied, then remove $(E,p,D)$ from $\mathcal{E}_q$ and return to Step 2. 
\item Output $\mathcal{E}_q$.
\end{enumerate}
\end{alg}

Note the step in Algorithm \ref{elkies} that has the highest time complexity is in Step 6, which has time complexity $\widetilde{O}(\log(f_q))$. Since $\mathcal{E}_q$ is of size $O(\log(f_q))$, the total time complexity of Algorithm \ref{elkies} is $\widetilde{O}(\log^2(f_q))$. For Algorithm \ref{eigenvalue}, Step 6 requires the heaviest computations, and it has time complexity $\widetilde{O}(\log^2(f_q))$. As the list $\mathcal{E}_q$ is of size $O(\log(f_q))$, the total time complexity for Algorithm \ref{eigenvalue} is $\widetilde{O}(\log^3(f_q))$.

\section{Some properties of Fibonacci numbers} \label{fibonacci}
Let $(f_n)_{n \geq 0}$ be the Fibonacci sequence given by $f_0 =0, f_1=1$ and $f_{n+2} = f_{n+1}+f_n$ for all $n \geq 0$. If a Fibonacci number $f$ is a rational prime, then we say that $f$ is a Fibonacci prime. As mentioned in the introduction, it is not known whether there is an infinite number of Fibonacci primes (see \cite{caldwell} for current commentary). One of the largest known Fibonacci prime is $f_{81839}$, which has $17103$ digits. However, heuristics regarding elliptic divisibility sequence (EDS) from \cite{eds} suggests it may be finite. Even though the sequence of Fibonacci numbers is not an EDS, with appropriate sign they are an EDS (see \cite{eds2}). Unfortunately, we could not go too afar in this path as the machinery is far beyond our capacity. We do find it to be interesting that Fibonacci primes can be written as a combination of an elliptic divisibility sequence (see Lemma \ref{sum}). In the following we will discuss some well-known results regarding Fibonacci numbers.

\begin{lem} Let $K$ be a field with characteristic not equal to $5$. Let $\alpha, \beta$ be the roots of the polynomial $G(X) = X^2-X-1$. Then the $n$th Fibonacci number is given by
\begin{equation}
f_n = \frac{\alpha^n - \beta^n}{\sqrt{5}}.
\end{equation}
\end{lem}
\begin{proof} This is an easy application of generating function. Here we will prove it directly.
Let $g_n = (\alpha^n - \beta^n)/\sqrt{5}$. It's clear that $g_0=0,g_1=1$. Furthermore, since $\alpha^2 = \alpha+1$, it follows that $\alpha^n = \alpha^{n-1}+\alpha^{n-2}$. Similarly, we have $\beta^n = \beta^{n-1}+\beta^{n-2}$. It is now straightforward to verify that $g_n = g_{n-1} + g_{n-2}$ for $n \geq 2$. Hence, $g_n = f_n$ for $n \geq 0$.
\end{proof}

\begin{lem} (Cassini's Identity) \label{square} For each integer $k \geq 0$, we have
\begin{equation}
f_{2k+1} = f_k^2 + f_{k+1}^2.
\end{equation}
\end{lem}

For Fibonacci prime $f_q$, the identity from Lemma \ref{square} readily provides us a square root of $-1 \pmod{f_q}$. Indeed, if we have $f_q = f_{(q+1)/2}^2 + f_{(q-1)/2}^2$, then a square root of $-1 \pmod{f_q}$ is given by
\begin{equation}
\sqrt{-1} \pmod{f_q} = f_{(q+1)/2}/f_{(q-1)/2} \pmod{f_q}.
\end{equation}
This observation allows us to quickly to test the primality of $f_q$ as seen in Step \ref{spre} of Algorithm \ref{mainalg}.

\begin{thm} \label{caseq} (See \cite{caseq}) Let $q$ be a rational prime. There exists integers $u,v$ such that
\begin{equation}
f_q = u^2 + qv^2.
\end{equation}
\end{thm}

\begin{lem} (See Chapter 1 of \cite{fibsum}) \label{sum} For each integer $n \geq 1$, we have
\begin{equation}
\sum_{k=1}^n f_{2k} = f_{2n+1}-1.
\end{equation}
\end{lem}

\begin{lem} \label{periodic}
The sequence $\{f_n \pmod{m}\}$ is periodic for any positive integer $m$.
\end{lem}
\begin{proof} This is clear by the Pigeonhole Principle.
\end{proof}

In fact by \cite{wall}, for each prime $\ell$, the period of the sequence $\{f_n \pmod{m}\}$  is the order of $\lambda$ in the field $\mathbb{F}_{\ell}[X]/(X^2-X-1)$, where $\lambda$ is a root of the polynomial $B(X) = X^2-X-1 \pmod{\ell}$. Using Quadratic Reciprocity, we have the following lemma:

\begin{lem}(\cite{wall}) \label{pisano} Let $\ell\neq 2,5$ be a prime. Then $\pi(\ell)$ is a divisor of $\ell-1$ if $\ell \equiv \pm 1 \pmod{10}$ and $\pi(\ell)$ is a quotient of $2(\ell+1)$ by an odd divisor if $\ell \equiv \pm 3 \pmod{10}$.
\end{lem}

\begin{lem} \label{primeindex} If $f_n \neq 3$ is a Fibonacci prime, then $n$ is a prime.
\end{lem}
\begin{proof}
This is due to the well-known fact that 
\begin{equation}
\gcd(f_n,f_m) = f_{\gcd(n,m)}.
\end{equation}
\end{proof}

\begin{lem} \label{divisible} For each prime $q$, we have the divisibility properties:
\begin{equation}\label{fact1}
q \mid f_{q-\legendre{5}{q}},
\end{equation}
\begin{equation}\label{fact2}
f_q \equiv \legendre{5}{q} \pmod{q}.
\end{equation}
\end{lem}

Lemma \ref{divisible} can be seen in \cite{will}, but we can prove easily using basic Galois Theory.  We will only prove \ref{fact1} as \ref{fact2} follows similarly. The result is clear if $q=5$ by Lemma \ref{periodic}, so assume $q \neq 5$. Let $K = \mathbb{F}_q[X]/(X^2-X-1)$. Note that the discriminant of $X^2-X-1$ is $5$. We will look at the cases $\legendre{5}{q} = 1$ and $\legendre{5}{q}=-1$ separately, that is, whether or not $5$ splits in $\mathbb{F}_q$. 

If $5$ splits in $\mathbb{F}_q$, then $K/\mathbb{F}_q$ is a degree $1$ extension. Hence, the Frobenius automorphism Frob$_p: K \rightarrow K$ defined by Frob$_q(x) = x^q$ is trivial, which implies that Frob$_q(\alpha) = \alpha^q = \alpha$ and Frob$_q(\beta) = \beta^q = \beta$.  It follows that Frob$_q(f_{q-1}) = f_{q-1}= (\alpha^q/\alpha - \beta^q/\beta)/\sqrt{5} = (1 - 1)/\sqrt{5} = 0$.

If $5$ does not splits in $\mathbb{F}_q$, then $K/\mathbb{F}_q$ is a degree $2$ extension, so the Frobenius automorphism is a conjugation map. It follows that Frob$_q(\alpha) = \alpha^q = \beta$. Hence, we have $f_{q+1} = (\alpha^q \alpha - \beta^q \beta)/\sqrt{5} = (\beta \alpha - \alpha \beta)/\sqrt{5} = 0$. The proof is complete.

Let $q > 3$ be a rational prime. Henceforth, let $C_q = C(-4f_q)$ and $B_q = B(-4f_q)$, as defined in Equation \ref{bound}. Recall that in Algorithm \ref{generators}, to find generators of the group $C_q$, we need to find the square roots $-f_q \pmod{\ell}$, for each prime $\ell \leq B_q$ such that $\legendre{-f_q}{\ell} = 1$. If $\ell \equiv 3 \pmod{4}$, we find in Section \ref{roots} that this is an easy computation. Now for the case $\ell \equiv 1 \pmod{4}$, note that $-f_q \equiv -f_r \pmod{\ell}$, where $r$ is the smallest non-negative integer such that $q \equiv r \pmod{\pi(\ell)}$. Hence, we do not need to compute $f_q$ to find a square root of $-f_q \pmod{\ell}$. Furthermore, as the primes $\ell \leq B_q$ are relatively small, computations of the square roots of $-f_q \pmod{\ell}$ are manageable. Noticeably, computing $\legendre{f_q}{\ell}$ is easy for primes $\ell < q$ by the periodicity and Lemma \ref{divisible}.

\begin{lem} \label{mod4} If $q > 3$ is prime, then $f_q \equiv 1 \pmod{4}$.
\end{lem}

\begin{proof}
If $f_q \equiv 3 \pmod{4}$, then $q \equiv 4 \pmod{6}$ by Lemma \ref{periodic}, which can not happen as $q$ is a prime. 
\end{proof}

\begin{obs} \label{2val} We note that given a probable prime $f_q$, the 2-valuation of $f_q-1$ is very small. If we write $f_q-1 = 2^e m$, where $m$ is odd, then $e \leq 6$. Here we have checked all the probable Fibonacci primes given at \cite{online}.
\end{obs}

Let $K = \mathbb{Q}(\sqrt{-f_q})$. By Lemma \ref{mod4}, $-f_q \equiv 3 \pmod{4}$, so $D_K = -4f_q$. By Genus Theory (see \cite[Theorem 6.1]{cox}), the number of elements of $C(\mathcal{O}_K)$ of order $2$ is $2^{t-1}-1$, where $t$ is the number of prime divisors of $D_K$. Since in our case $D_K = -4f_q$, we have $t-1 \geq 1$. Hence, $C(\mathcal{O}_k) \cong C_q$ always has even order. It would be of interest to study the 2-sylow subgroup of $C_q$.

\begin{prop} (See also \cite[Corollary 18.6]{parity}) \label{parity} Suppose $f_q$ is a prime for $q > 3$. The $2$-sylow subgroup of $C_q$ is cyclic.
\end{prop}

\begin{proof}
Suppose $f_q$ is a Fibonacci prime. Then the number of prime divisors of $-4f_q$ is $2$. Hence, $C_q$ has $2^{2-1}-1 = 1$ element of order $2$. As $C_q$ has exactly one element of order $2$,  it must be the case that the $2$-sylow subgroup of $C_q$ is cyclic.
\end{proof}

\begin{prop} \label{cyclic} The probability that $C_q$ is cyclic is at least $97\%$.
\end{prop}

\begin{proof}
Heuristics from \cite{cl} by Cohen and Lenstra states (conjecturally) that the odd part of a class group $C(D)$ is cyclic for at least $97\%$ of the time (see also Conjecture 5.10.1 of \cite{cohen1}). Since the $2$-sylow subgroup of $C_q$ is always cyclic, we have the desired result.
\end{proof}

\begin{prop}(See \cite[Corollary 19.6]{parity}) \label{2sylow} Let $q$ be a prime and $E = \mathbb{Q}(\sqrt{-q})$. Then the 2-Sylow subgroup of $C(D_E)$ has order $2$ if and only if $q \equiv 5 \pmod{8}$.
\end{prop}

\begin{cor} The 2-sylow subgroup of $C_q$ has order $2$ exactly when $q \equiv 5,7 \pmod{12}$.
\end{cor}

\begin{proof}
By Lemma \ref{2sylow}, the $2$-sylow of $C_q$ has order $2$ exactly when $f_q \equiv 5 \pmod{8}$. By Lemma \ref{pisano}, $f_q \equiv 5 \pmod{8}$ exactly when $p \equiv 5,7,8 \pmod{12}$, but $q \equiv 8 \pmod{12}$ can not happen as $q$ is a prime.
\end{proof}

Having a nontrivial 2-Sylow allows us to factor $f_q$ if it is not prime using the Shank's Class Group Method. All of the elements of order $2$ are of the form $(a,a,c),(a,0,c),$ and $(a,b,a)$. For example, if we have the form $(a,a,c)$, then we have $-4f_q = a^2-4ac = a(a-4c)$, which implies $f_q = (a/2)(a/2 - 2c)$, though the factorization may be trivial. Obtaining an element of order $2$ is relatively straightforward if we can easily compute the class number $h$ of $C_q$. Indeed, factor $h$ as $h=2^ek$, where $k$ is odd. Compute an arbitrary prime form $F_{\ell}$ of $C_q$ using Algorithm \ref{generators}, then the element $F_{\ell}^k$ is in the 2-Sylow subgroup of $C_q$.  Of course, the difficulty lies in computing $h$, which is extremely difficult for large $f_q$. Since $C_q$ is highly cyclic, we could in theory probabilistically compute its order using Atkin's Variant (see \cite[pp. 252--261]{cohen1}); however, its time complexity is sub-exponential.

Assume by some miracle that we are able to compute the class number of $C_q$, then we can easily obtain an element of order $2$. As mentioned above, the class group has a high probability of being cyclic, which implies $C_q$ has exactly one element of order $2$ most of the times. Hence, each element of order $2$ allows us a way to factor $f_q$, and even if multiple elements of order $2$ yield trivial factorization, the fact that we have more than one element of order $2$ hints of the fact that $f_q$ may be composite.

We will now describe a variant of Shanks's Class Group Method, but we will not use it test probable Fibonacci primes as it is not practical with known machinery. We will avoid this \textit{factoring in the dark ages} (see \cite[ch. 8]{cohen1}) and provide much better primality tests in Section \ref{bettertests}.

\begin{alg} \label{primetest} Variant of Shanks's Class Group Method
\begin{enumerate}[nolistsep]
\item Find the set of prime generators $\mathcal{F}$ for $C_q$ using Algorithm \ref{generators}.
\item Use Algorithm \ref{generators} to find prime forms of $C_q$ and obtain elements of order $2$. Also let $n$ be the number of trials until the first non-quadratic residue $z$ is found.
\item If $n \geq 50$, then $f_q$ is likely composite.
\item If an element of order $2$ is found, use it to factor $f_q$. If the factorization is nontrivial, then $f_q$ is composite.
\item If two or more distinct elements of order $2$ are found, then $f_q$ is definitely not a prime for $q \equiv 5,7 \pmod{12}$, otherwise $f_q$ is likely composite.
\item Verify that $|\mathcal{F}| \approx B(D)/2$, else $f_q$ is likely not prime.
\item If $f_q$ passes the previous steps, then $f_q$ is a probable prime.
\end{enumerate}
\end{alg}

\section{Exceptional Cases} \label{exceptional}

We have from Lemma \ref{square} that 
\begin{equation}
4f_q = (2f_{(q+1)/2})^2 + 4f_{(q-1)/2}^2,
\end{equation}
so $f_q$ splits completely in the Hilbert class field of $\mathbb{Q}(\sqrt{-1})$, which is itself since the class number is $1$. This is not all surprising since $\legendre{-1}{f_q} = 1$. Generalizing this, we have that $f_q$ splits completely in the Hilbert class field $H_D$ of any imaginary quadratic field $\mathbb{Q}(\sqrt{D})$ with class number $1$ for which $\legendre{D}{f_q} = 1$. It is well-known that the set of all $d < 0$ such that the field $\mathbb{Q}(\sqrt{d})$ has class order $1$ is the set of Heegner numbers 
\begin{equation} \label{order1}
H=\{-1,-2,-3,-7,-11,-19,-43,-67,-143\}.
\end{equation}
For example, if $f_q \equiv 1 \pmod{8}$, then 
\begin{equation}
4f_q = (2x)^2 + 8y^2
\end{equation}
for some integers $x,y$. Step \ref{sdisc} skips over $d=-1,-2,-3$ and $-7$. We will use these cases to quickly test the primality of $f_q$.

As mentioned in Section \ref{fibonacci}, computing $\legendre{\ell}{f_q} = \legendre{\ell}{f_q}$ can be computed quickly with primes $\ell < q$ due to the fact the sequence $(f_n \pmod{m})_{n\geq 0}$ is periodic. For example, $\legendre{3}{f_q} = 1$ exactly when $f_q \equiv 1 \pmod{3}$, which can be quickly determined. Indeed, when $m=3$, the period is $8$ and $f_n \equiv 1$ exactly when $n \equiv 1,2,7 \pmod{8}$. However, the index $q$ is prime, so it follows that $f_q \equiv 1 \pmod{3}$ exactly when $q \equiv \pm 1 \pmod{8}$. Hence, $4f_q  = x^2 + 3y^2$ for some positive integers $x,y$ if and only if $q \equiv \pm 1 \pmod{8}$. Hence, for each $d \in H$, $\legendre{d}{f_q}$ is sufficient for $f_q$ to be a norm in $\mathbb{Q}(\sqrt{-d})$, and as observed above, $\legendre{d}{f_q}$ is easily computed. 

Again from Lemma \ref{caseq}, we have
\begin{equation}
f_q = x^2+y^2q
\end{equation}
for some positive integers $x,y$. Moreover, from Lemma \ref{divisible}, it follows that
\begin{equation}
x^2 \equiv \legendre{5}{q} \pmod{q}.
\end{equation}
Hence, we can take $x=1$ or $x = \sqrt{-1} \pmod{q}$, and apply the Cornacchia's Algorithm \ref{corn} to find find positive integers $x,y$ so that $4f_q = (2x)^2+ y^2(4q)$, which implies $f_q$ always split in the Hilbert class field of $\mathbb{Q}(\sqrt{-q})$.

From these simple observations, we have the following algorithm to test the primailty of $f_q$.

\begin{alg}(Exceptional Cases Test) \label{excases} This algorithm uses ECPP to test the primality of $f_q$.
\begin{enumerate}[nolistsep]
\item Let $d = [-1,-2,-3,-7,-q]$.
\item Let $m=0$.
\item If $m < 5$, then let $\ell = d[m]$, else go to the final step.
\item Compute $s=\legendre{\ell}{f_q}$ and go to the next step.
\item If $s = -1,0$, then increase $m$ by $1$ and return to Step 3, else go to the next step.
\item Let $K_{\ell} = \mathbb{Q}(\sqrt{\ell})$ and Let $D_{\ell}$ be the discriminant of $K_{\ell}$.
\item Find positive integers $x,y$ for which $4f_q = x^2+y^2|D_{\ell}|$ if such exist. If no such $x,y$ exist, then increase $m$ by $1$ and return to Step 3.
\item Determine if $p = f_q+1\pm x$ has a prime divisor $q > (f_q^{1/4}+1)^2$. If it is too difficult to determine, increase $m$ by $1$ and return to Step 3, else go to the next step.
\item Let $r$ be a root of $H_D(X) \pmod{f_q}$ and go to the next step. If no such $r$ exists, then $f_q$ is composite and we end the algorithm.
\item If $r=0$ or $r=1728$, take $E/(\mathbb{Z}/f_q\mathbb{Z})$ to be the curve $Y^2 = X^3+1$ or $Y^2=X^3+X$, respectively. If $r \neq 0,1728$, take the curve $Y^2=X^3+ax-a$, where $a = 27r/(4(1728-r)) \pmod{f_q}$. Take $P=(-1,0)$ or $P=(0,0)$ if $r=0$ or $r=1728$, respectively. Take $P=(1,1)$ if $r\neq 0,1728$.
\item Apply ECPP (Theorem \ref{ecpp}) to the curve $E$ with the point $P$.
\item If $f_q$ is confirmed to be prime by ECPP, then we stop the algorithm, else increase $m$ by $1$ and return to Step 3.
\item The integer $f_q$ is a probable prime.
\end{enumerate}
\end{alg}

\section{Well-known primality tests} \label{bettertests}
In this section we will discuss a number of primality tests. The computations done in these tests are some of the same computations needed in the construction of an elliptic curve of order $f_q$ over some finite field, so there is no loss of computations in performing these tests. This observation is noted in \cite{broker2} as well.

Assuming GRH, Bach \cite{bach} has shown that if $p > 1000$, then there exists a quadratic non-residue $z$ modulo $p$ less than $2 \log^2(p)$. Furthermore, by the Chebotarev Density Theorem, half of the primes in the interval $[1,2 \log^2{p}]$ are quadratic residues modulo $p$. Hence, we have the following crude test for primality:

\begin{thm} \label{densitytest} (Density Test) Let $p$ be a probable prime. Compute $\legendre{\ell}{p}$ for all primes $\ell \leq 2 \log^2(p)$. If it takes about $50$ trials to find a quadratic non-residue $n$ or if the number of quadratic residues is not approximately $\log^2(p)$, then $p$ is likely composite.
\end{thm}

\begin{thm} \label{fermat} Let $p$ be a probable prime, and write $p-1 = 2^e d$, where $d$ is odd. If we can find an integer $a$ such that
\begin{equation}
a^d \equiv 1 \pmod{p}
\end{equation}
and 
\begin{equation}
a^{2^r d} \not \equiv -1 \pmod{p}
\end{equation}
for all $0 \leq r < e$, then $p$ is not prime.
\end{thm}

Notice that Theorem \ref{fermat} is just the contrapositive of Fermat Little Theorem. We will now describe the Rabin-Miller Primality Test following \cite{cohen1}.

\begin{alg} \label{rm}(Rabin-Miller Primality Test) Let $p$ be a probable prime, and write $p-1 = 2^e d$, where $d$ is odd.
\begin{enumerate}[nolistsep]
\item Choose $20$ random integers in the interval $[2,p-1]$, and  store them in a set $W(p)$.
\item \label{pickwitness} If $W(p)$ is empty, go to the final step, else pick $a \in W(p)$ and remove $a$ from $W(p)$.
\item Let $k=0$.
\item Compute $b=a^d \pmod{p}$. 
\item \label{rmsquare} If $b = \pm 1$, then return to Step \ref{pickwitness}, else increase $k$ by $1$ and go to the next step.
\item If $k = e$ and $b \neq -1 \pmod{p}$, then $p$ is composite, else go to the next step.
\item Compute $b^2 \pmod{p}$, and return to Step \ref{rmsquare}.
\item The prime $p$ is a probable prime.
\end{enumerate}
\end{alg}

It is clear the the Rabin-Miller Primality Test has time complexity $\widetilde{O}(\log^2(p))$ since we are computing only a few exponentiations. 

\begin{thm} \label{ecpp}(Elliptic Curve Primality Proving) Let $p > 6$ be a probable prime. Let $E/(\mathbb{Z}/p\mathbb{Z})$ be an elliptic curve of order $kq$, where $q$ is a prime such that $q > (p^{1/4} +1)^2$. If there exists a point $P$ on $E$ such that $kqP=0$ and $kP$ is defined and not equal to $0$, then $p$ is a prime.
\end{thm}

We repeatedly use ECPP in Step \ref{secpp} of Algorithm \ref{mainalg}. In Step \ref{secpp}, even when $p=f_q+1 \pm x$ fails the Rabin-Miller Primality Test, we can still use it to test the primality of $f_q$ using ECPP. If $f_q$ is confirmed to be prime, then Step \ref{scheckorder} does indeed confirm the primality of $p$.

For a discussion of Elliptic Curve Primality Proving see \cite{am} and \cite[ch. 14]{cox}.

\bibliography{ellipticfib}
\bibliographystyle{alpha}

\end{document}